\documentclass[12 pt]{amsart}

\usepackage{times}
\usepackage{geometry}
\usepackage{amssymb}
\usepackage{latexsym,amssymb,  amsmath, amscd, amsfonts}
\usepackage{graphicx}
\usepackage[percent]{overpic}
\usepackage{pdfsync}
\usepackage{units}
\usepackage{hyperref} 
\usepackage{euscript}
\usepackage{multicol}
\usepackage{epstopdf}
\usepackage{paralist}

\newtheorem{theorem}{Theorem}
\newtheorem{lemma}[theorem]{Lemma}
\newtheorem{proposition}[theorem]{Proposition}
\newtheorem{corollary}[theorem]{Corollary}

\theoremstyle{definition}
\newtheorem{definition}[theorem]{Definition}

\newtheorem{conjecture}[theorem]{Conjecture}
\newtheorem{const}[theorem]{Construction}
\newtheorem{remark}[theorem]{Remark}
\newtheorem{open}[theorem]{Open Question}

\newcommand{\R}{\mathbb{R}}

\def\Lk{{\operatorname{Lk}}}
\def\Tw{{\operatorname{Tw}}}
\def\Wr{{\operatorname{Wr}}}
\def\Rib{{\operatorname{Rib}}}
\def\Cr{{\operatorname{Cr}}}
\def\Len{{\operatorname{Len}}}
\def\Rop{{\operatorname{Rop}}}
\newcommand{\ds}{\displaystyle}

\begin{document}

\title[Bounded ribbonlength for knot families and multi-twist M\"obius bands]{Bounded ribbonlength for knot families and \\ multi-twist M\"obius bands}
\author[E. Denne]{Elizabeth Denne}
\address{Elizabeth Denne: Washington \& Lee University, Department of Mathematics, Lexington VA}
\email[Corresponding author]{dennee@wlu.edu}
\urladdr{https://elizabethdenne.academic.wlu.edu/}
\author[T. Patterson]{Timi Patterson}
\address{Timi Patterson: Washington \& Lee University}
\email{tpatterson@mail.wlu.edu}
\date{\today}
\makeatletter								
\@namedef{subjclassname@2020}{%
  \textup{2020} Mathematics Subject Classification}
\makeatother


\begin{abstract}
Take a thin, rectangular strip of paper, add in an odd number of half-twists, then join the ends together. This gives a multi-twist paper M\"obius band. We prove that any multi-twist paper M\"obius band can be constructed so the aspect ratio of the rectangle is $3\sqrt{3}+\epsilon$ for any $\epsilon>0$. We could also take the thin, rectangular strip of paper and tie a knot in it, then join the ends and fold flat in the plane. This creates a folded ribbon knot. We apply the techniques used to prove the multi-twist paper M\"obius band result to $(2,q)$ torus knots and twist knots. We prove that any $(2,q)$-torus knot  can be constructed so that the folded ribbonlength $\leq 13.86$. We prove that any twist knot can be constructed so that the folded ribbonlength is $\leq 17.59$. Both of these results give the lower bound for the ribbonlength crossing number problem which relates the infimal folded ribbonlength of a knot type $[K]$ to its crossing number $\Cr(K)$.  That is, we have shown $\alpha=0$ in the equation $c\cdot \Cr(K)^\alpha \leq \Rib([K])$, where $c$ is a constant. 
\end{abstract}

\maketitle

\section{Introduction} \label{sect:intro}

Imagine taking a thin rectangular strip of paper and simply joining the ends creating a cylinder (or annulus). You could first add a half-twist, then join the ends to create a M\"obius band. You could also add in any number of half-twists before joining the ends together. If there are an odd number of half-twists you create  a {\em multi-twist paper M\"obius band}, if there are an even number of half-twists, you create a {\em multi-twist paper annulus}.  We pause to be clarify that the M\"obius bands and annuli arise from a smooth isometric embedding of a flat M\"obius band or cylinder in $\R^3$ (For the formal definition see \cite{RES-Mob} and Section~\ref{sect:mob-defn}.)

Alternatively, you could take the thin, rectangular strip of paper and tie a knot in it before joining the ends. In each of these scenarios, the center line of the rectangular strip is a knot in $\R^3$, while the rectangular strip gives a framing of the knot. The twisted annuli and M\"obius bands are thus unknots with a particular framing.  In this paper we will also consider {\em folded ribbon knots}, which occur when the knotted rectangular strips of paper have been folded flat in the plane. Figure~\ref{fig:Mob-trefoil} shows two different M\"obius bands and, on the right, a folded ribbon trefoil knot.  Working in the plane simplifies many of the computations. We can think of a folded ribbon knot as a piecewise-linear immersion of an embedded-knotted M\"obius band or annulus into the plane where the fold lines correspond to the singularities of the immersion. (See \cite{Den-FRS} for a survey article on folded ribbon knots and Section~\ref{sect:defn} for the formal definition.)

\begin{center}
    \begin{figure}[htbp]
        \begin{overpic}{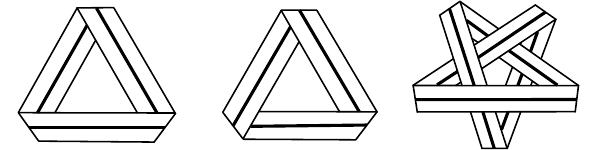}
        \end{overpic}
        \caption{On the left, a M\"obius band made of one half-twist, and in the center one made of three half-twists. On the right, a folded ribbon trefoil knot. }
        \label{fig:Mob-trefoil}
    \end{figure}
\end{center}

Over the years many people have tried to find the {\em infimal aspect ratio} of the rectangular strip of paper used to tie any multi-twist M\"obius band or annulus. In the language of folded ribbon knots, this is the  same as the {\em ribbonlength problem}: to find the infimal folded ribbonlength of a particular knot or link type. The {\em folded ribbonlength} of a folded ribbon knot is the infimal ratio of the length of the knot to the width of the ribbon.

In 2024, Richard Schwartz \cite{RES-Mob} proved a long standing conjecture of Ben Halpern and Charles Weaver \cite{HW}, that any smoothly embedded paper M\"obius band must have aspect ratio greater than $\sqrt{3}$. This paper also includes a review of the history of this problem.  A secondary result of \cite{RES-Mob}  is that any sequence of embedded paper M\"obius bands with aspect ratio converging to $\sqrt{3}$ must converge to a limiting example which is a three fold wrapping of an equilateral triangle. If you stare at the left image in Figure~\ref{fig:Mob-trefoil} you might be able to imagine this limiting example. In a second paper, Schwartz \cite{RES-Mob2} gives an explicit estimate for this convergence. 

What about embedded cylinders or multi-twist constructions? In 2023, Brienne Brown and Schwartz \cite{RES-Brown} gave two different constructions showing a M\"obius band made of three half-twists has aspect ratio greater than $3$. Note that the center image in Figure~\ref{fig:Mob-trefoil} is such a M\"obius band, but the infimal aspect ratio of this geometric configuration is $3\sqrt{3}$ (see \cite{DKTZ}). In 2025, Schwartz and Noah Montgomery \cite{RES-Mont} wrote a joint paper on the {\em twisted paper cylinder}. This is an isometric embedding of a flat cylinder in $\R^3$ such that the images of the boundary components are linked. An annulus made with two half-twists is an example of such a twisted paper cylinder. They proved that any twisted paper cylinder has aspect ratio greater than 2 and that this bound is sharp. 

More recently, Aidan Hennessey \cite{Hen} showed that there are multi-twist paper M\"obius bands and annuli with as many twists as desired but with aspect ratio less than 8.  In \cite{RES-Mont}, Schwartz refers to unpublished work of Jan Neinhaus that shows that for the M\"obius band, the aspect ratio of $3\sqrt{3}+\epsilon$ suffices for any $\epsilon >0$. One of the goals of this paper is to provide a published proof of this result (see Theorem~\ref{thm:min-angle}).

A secondary goal of this paper is to apply the constructions used for the multi-twist M\"obius band to folded ribbon knots. Like the embedded paper M\"obius bands, folded ribbon knots in the form of ``tight'' folded ribbon trefoil knots have appeared in recreational mathematics \cite{CR, John, Wel}. Since 2004, there have been a number of papers  which tackle the  folded ribbonlength problem by constructing a folded ribbon knot for an infinite family of knots. These constructions give upper bounds on the infimal folded ribbonlength for this family of knots. The ribbonlength problem can be viewed as a 2-dimensional version of the {\em ropelength problem} which asks for the minimum amount of rope needed to tie a knot in a rope of unit diameter.  (See for instance \cite{BS99,CKS,DDS, DE,GM,lsdr}.) Knotted ribbon shapes also appear in DNA structures and some of these are circular \cite{Flap, DNA1}. The folding of two and three-dimensional structures in ribbons has also appeared in robotics \cite{RobFold}, and in genetics with a ribosomal walking robot \cite{RiboRobot}.

A separate, but equally interesting, open problem is to relate the infimal folded ribbonlength of a knot type $[K]$ to its crossing number $\Cr(K)$.  The {\em ribbonlength crossing number problem} asks for positive constants $c_1, c_2, \alpha, \beta$ such that
\begin{equation}
c_1\cdot \Cr(K)^\alpha \leq \Rib([K]) \leq c_2\cdot \Cr(K)^\beta.
\label{eq:crossing}
\end{equation}
There is a long standing conjecture by Yuanan Diao and Rob Kusner, that $\alpha=\frac{1}{2}$ and $\beta=1$ in Equation~\ref{eq:crossing}. In 2017, Grace Tian  \cite{Tian-A} proved that $\beta\leq 2$. In 2021, the first author in  \cite{Den-FRC} proved that $\beta\leq \frac{3}{2}$, and with coauthors in  \cite{Den-FRF} proved that $\alpha\leq \frac{1}{2}$ for both knots and links. In 2024, Hyoungjun Kim, Sungjong No, and Hyungkee Yoo~\cite{KNY-Lin} gave a solution for one part of this conjecture  by proving that that $\beta\leq 1$. Specifically, they proved that for any knot or link, 
\begin{equation}\Rib(K)\leq 2.5\Cr(K)+1.
\label{eq:bound}
\end{equation}
 In this paper, we will prove the second part of this conjecture by showing that $\alpha =0$. Specifically, in Theorem~\ref{thm:torus-bound} we will prove that any $(2,q)$-torus knot $T(2,q)$ can be constructed so that the folded ribbonlength $\Rib([T(2,q)])\leq 13.86$. We give a second proof by of this result by looking at twist knots $T_n$.  Theorems~\ref{thm:twist-bound} and~\ref{thm:twist-even} tell us $\Rib([T_n])\leq 17.59$ when $n$ is odd, and
$\Rib([T_n])\leq 15.86$ when $n$ is even.
 
 Finally, we remind the reader that there have been many papers \cite{Den-FRC, Den-FRF, DKTZ, HHKNO, KMRT, KNY-TwTorus, KNY-2Bridge} which give upper bounds on the ribbonlength in terms of crossing number for various families of knots. These bounds are often better than the bounds in Equation~\ref{eq:bound} for specific families of knots and also for knots with low crossing number.  We will see that there is often a  difference between the folding techniques used to find the best upper bounds on small crossing knots versus the techniques used for infinite families.  

The work in this paper is structured as follows. In Section~\ref{sect:defn} we give the formal definition for folded ribbon links and relate this to the work for embedded paper M\"obius bands.   We also carefully discuss the framing of folded ribbon links and how this is related to folded ribbonlength. The key idea is that a folded ribbon knot can be well-approximated by embedded paper M\"obius bands or annuli (Lemma~\ref{lem:approx}). This means that we can use folded ribbon knot constructions and folded ribbonlength to give upper bounds on the infimal aspect ratio of embedded paper M\"obius bands.

 In Section~\ref{sect:mob} we introduce the idea of an accordion fold which is a key part of our Construction~\ref{const:Mobius} for a folded ribbon multi-twist M\"obius band. In Theorem~\ref{thm:min-angle} and Corollary~\ref{cor:mmmb} we show that the aspect ratio of a paper multi-twist M\"obius band is bounded above by $3\sqrt{3}+ \epsilon$ for any $\epsilon>0$. This is a uniform upper bound that does not depend on the number of half-twists. We also observe that different constructions are needed for multi-twist paper bands which have small numbers ($\leq 6$) of half-twists. In Appendix~\ref{sect:6twist} we review existing constructions for $n=2,3,4$ half-twist paper bands. We then show how these can be combined to give a construction for a 5 half-twist paper M\"obius band with aspect ratio~5 and a 6-half-twist paper annulus with aspect ratio 4.

In the remaining sections of the paper, we apply the accordion fold construction to folded ribbon knots and links. In Sections~\ref{sect:torus} and \ref{sect:twist}, we look at the $(2,q)$-torus knots and twist knots. Both of these knots are constructed from a number of half-twists. The ends of the half-twists are then joined in specific ways to create the knots (see Figures~\ref{fig:torus} and~\ref{fig:twist-knots}). The idea is that like the multi-twist M\"obius band, the half-twists can be made with finite ribbonlength, no matter how many half-twists are added. There is then a second finite amount of ribbon needed to join the ends of the half-twists to construct the torus and twist knots. These constructions prove that any $(2,q)$-torus knot and twist knot with arbitrarily high crossing number can be made with finite folded ribbonlength (see Theorems~\ref{thm:torus-bound}, \ref{thm:twist-bound}, and \ref{thm:twist-even}). These results reveal just how different folded ribbon knots are from knots made out of physical rope. We know that the ropelength of any knot $K$ is bounded below by the crossing number: $\Cr(K)^{3/4}\leq \Rop(K)$.
 
We pause here to encourage the reader to find long strips of paper with which they can recreate the constructions of multi-twist M\"obius bands and folded ribbon knots found in this paper. Many of the constructions will make more sense that way.

\section{Folded Ribbon Knots and Multi-twist M\"obius bands}\label{sect:defn}

\subsection{Folded ribbon links and folded ribbonlength}\label{sect:rib}
Following Lou Kauffman \cite{Kauf05}, we create a folded ribbon link by taking a polygonal link diagram\footnote{Polygonal links are made of a finite number of straight line segments. We use standard the definitions and theory of knots and links as found in texts like \cite{Adams, Crom, JohnHen}.}  $L$ and placing two parallel lines at an equal distance from each edge (with one on each side).  At each vertex of the link diagram, we add a fold line which is perpendicular to the {\em fold angle} $\theta$ at the vertex (here $0\leq\theta\leq \pi$). This results in a ribbon with width $w>0$,  denoted by $L_w$. We note that the fold lines act like a mirror that reflect adjacent edges of the link diagram into each other.  A comparison of a polygonal trefoil knot and its corresponding folded ribbon knot can be seen in Figure~\ref{fig:mob-trefoil}. Since we are considering the folded ribbon link in the plane, we assign continuous over- and under-information to the ribbon which is consistent with the crossing information in the knot diagram. Complete details of this definition can be found in \cite{Den-FRS, DKTZ}.  In \cite{Den-FRF}, we prove that every polygonal link diagram has a folded ribbon link provided that the width $w$ is small enough.

\begin{center}
    \begin{figure}[htbp]
        \begin{overpic}{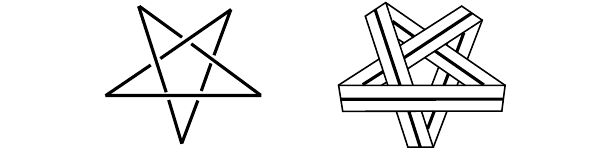}
        \end{overpic}
        \caption{We create a folded ribbon trefoil knot (right) by adding a ribbon with fold lines to a polygonal trefoil knot diagram (left).}
        \label{fig:mob-trefoil}
    \end{figure}
\end{center}

Recall that the {\em folded ribbonlength} of a folded ribbon link $L_w$ is defined as the ratio of the length of the link $L$ to its width $w$, or 
$\ds \Rib(L_w) = \nicefrac{\Len(L)}{w}.$
As discussed in Section~\ref{sect:intro}, the ribbonlength problem seeks to find the ``shortest piece of ribbon'' needed to tie a folded ribbon knot for a particular knot or link type which we denote by $[L]$. More formally, 
$$\Rib([L]) = \inf_{L_w \in [L]} \text{Rib}(L_w). $$

For example, the ribbonlength of the folded ribbon trefoil knot $K$ shown in Figure~\ref{fig:mob-trefoil} is $\Rib(K_w)=5\cot(\pi/5)\leq 6.89$. However, there are three separate constructions (see \cite{CDPP, Den-FRF}) of a folded ribbon trefoil knot which have $\Rib(K_w)=6$. We conjecture that the infimal ribbonlength of the trefoil is $\Rib([K])=6$. In addition, the unknot $U$ can be constructed from a diagram with two overlapping edges. The corresponding folded ribbon unknot is an annulus, and as the length of these edges can be made as small as we like, we have $\Rib([U])=0$.

 Observe that locally, a fold in a piece of ribbon consists of two overlapping triangles, while a crossing consists of two overlapping parallelograms. There is a lot of symmetry in these figures, and a complete description of the local geometry and trigonometry of folds and crossings can be found in \cite{Den-FRF}.  We end this subsection with the following remark.

\begin{remark}\label{rmk:length} 
Observe that the ribbonlength of any square ribbon segment is $1$. Now, fold a piece of ribbon so that the fold angle $\theta = \nicefrac{\pi}{2}$. This creates two overlapping right isosceles triangles, each with a ribbonlength of $\nicefrac{1}{2}$. This $\nicefrac{\pi}{2}$-fold that has a total ribbonlength of $1$. \end{remark}

\subsection{Folded ribbon link equivalence}
We review some features of folded ribbon links as found in \cite{Den-FRS, Den-FRF, DKTZ}. We create an {\em oriented} folded ribbon link by assigning a direction of travel along the knot diagram. Once a folded ribbon link is oriented, we can classify each of its folds as either an {\em underfold} or an {\em overfold}, as shown in Figure~\ref{fig:uofold}. Thus the {\em folding information} $F$ is an essential part of a folded ribbon link $L_w$, and sometimes we emphasize this by using the notation $L_{w,F}$. We can capture the framing of the folded ribbon around the link diagram by computing the {\em ribbon linking number}. This is the linking number between the knot diagram and one boundary component of the ribbon\footnote{The linking number of an oriented two component link $A\cup B$ is one half of the sum of the signs of the crossings between components $A$ and $B$ in any link diagram of $A\cup B$. See for instance \cite{Adams, Crom, JohnHen}.}, and is denoted by $\Lk(L_{w,F} )$.

\begin{center}
    \begin{figure}[htbp]
        \begin{overpic}{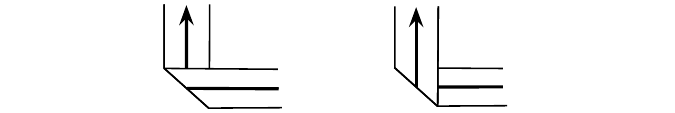}
        \end{overpic}
        \caption{The distinction between an underfold (left image) and an overfold (right image) depends on the orientation of the link. Image re-used with permission from \cite{DKTZ}.}
        \label{fig:uofold}
    \end{figure}
\end{center}

When trying to find the infimal ribbonlength, we need to think about which equivalence class of folded ribbon links we are using. We can optimize the ribbonlength and only consider the link type: this is {\em diagram equivalence} as found Section~\ref{sect:rib}. Alternatively, we can take the framing of the knot into consideration. Thus two oriented folded ribbon links are {\em (folded) ribbon equivalent} provided they
\begin{compactitem}
\item are diagram equivalent,
\item have the same ribbon linking number, and
\item are both topologically equivalent either to a M\"obius band or an annulus when
considered as ribbons in $\R^3$.
\end{compactitem}
An easy check will show the reader that when a polygonal knot diagram has an even number of edges, then the corresponding folded ribbon knot is a topological annulus. If a polygonal knot diagram has an odd number of edges, then the corresponding folded ribbon knot is a topological M\"obius band.  Now consider an unknot $U$ which is a topological M\"obius band. The results of Schwartz \cite{RES-Mob} and \cite{DKTZ} mean we now know that $\Rib([U_w]_{\text{M\"ob}})=\sqrt{3}$. Finally, we observe that the folded ribbon trefoil knot in Figure~\ref{fig:mob-trefoil} is a topological M\"obius band. This leads us to the following conjecture.
\begin{conjecture} The infimal folded ribbonlength of a folded ribbon trefoil knot $K$ which is a topological M\"obius band is
$$\Rib([K]_{\text{M\"ob}})=5\cot(\pi/5)\leq 6.89.$$ 
\end{conjecture}

\subsection{Multi-twist M\"obius Bands and Annuli}\label{sect:mob-defn}

Following Schwartz \cite{RES-Mont, RES-Mob}, we formally define an {\em embedded paper M\"obius band} of aspect ratio $\lambda$ as a $C^\infty$-smooth, arc-length preserving embedding of $M_\lambda$ into $\R^3$. Here, $M_\lambda$ is the flat M\"obius band obtained by identifying the length-1 sides of a $1\times\lambda$ rectangle: 
$$M_\lambda = ([0,\lambda]\times[0,1])/((0,y)\sim(\lambda, 1-y)).$$
Similarly, an {\em embedded paper cylinder of aspect ratio $\lambda$} is a $C^\infty$-smooth, arc-length preserving embedding of the flat cylinder $\Gamma_\lambda=([0,\lambda]\times[0,1])/((0,y)\sim(\lambda, y))$ into $\R^3$.  We refer to these embeddings collectively as {\em smooth paper bands}.

We pause to note that smooth paper bands are a more general class of objects than folded ribbon knots. Many paper bands can be folded flat in the plane forming folded ribbon knots, for example the two M\"obius bands found in Figure~\ref{fig:Mob-trefoil}. However, there are smooth paper bands that cannot be folded flat in the plane. For example, the M\"obius band made of three half-twists forming a cup-like shape found in \cite{RES-Brown}. 

The language of ribbon linking number helps us formalize the definition of multi-twist M\"obius bands and annuli. Following Hennessey \cite{Hen}, the centerline of a smooth paper band is the image of $[0,\lambda]\times\{\nicefrac{1}{2}\}$. We define an {\em $n$ half-twist paper band} to be:
\begin{compactitem}
\item A paper M\"obius band for which the boundary and centerline have linking number $\pm n$ (for $n$ odd).
\item A paper annulus for which one of the boundaries and the centerline have linking number $\pm n/2$ (for $n$ even).
\end{compactitem}

One of the other key results in \cite{Hen} is that there is a folded ribbon unknot which is a 
\begin{compactitem}
\item topological M\"obius band with ribbon linking number $\pm n$ (for $n$ odd).
\item topological annulus with ribbon linking number $\pm n/2$ (for $n$ even).
\end{compactitem}

There are thus two different, but related objects of study. To help keep things straight we also call an $n$ half-twist paper M\"obius band/annulus a {\em (multi-twist) paper M\"obius band/annulus}. We will call a framed folded ribbon unknot a {\em (multi-twist) folded ribbon M\"obius band or annulus} as appropriate. In \cite{Den-FRLU}, the first author with Troy Larsen gave the first upper bounds on the folded ribbonlength of folded ribbon annuli. 

Hennessey achieves his upper bounds on the aspect ratios of smooth paper bands by proving that a folded ribbon unknot (with appropriate ribbon linking number) can be well approximated by smooth paper bands (see \cite{Hen} Section 3.3). We can generalize his arguments to show that any folded ribbon knot can be well approximated by a smooth paper band. After all, folded ribbon knots are either topologically a M\"obius band or annulus.  

\begin{center}
    \begin{figure}[htbp]
        \begin{overpic}[height=1.25in]{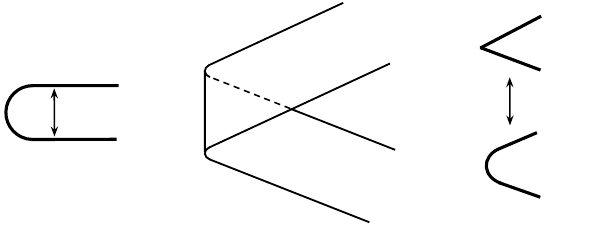}
        \put(10,18){$\epsilon$}
        \end{overpic}
        \caption{On the left and center, a side and top view of a pseudo-fold that replaces a fold line. On the right, a short piece of the knot diagram near a vertex is approximated by a short curved segment.}
        \label{fig:pseudo}
    \end{figure}
\end{center}

To do this, we first compute the folded ribbonlength of the folded ribbon knot $K_w$ in question. Next, we separate the layers of the folded ribbon knot by some distance $\epsilon$ near each vertex as shown on the left in Figure~\ref{fig:pseudo}. Note that we cannot assume that each edge of the knot $K$ is at a constant height\footnote{For example, the trefoil knot in Figure~\ref{fig:mob-trefoil} cannot be constructed with each edge at a constant height.}. However, the difference between the length of the perturbed edge and the original is small and we can bound the error by a constant multiple of $\epsilon$. We then replace the fold lines with {\em pseudo-folds}. These are illustrated on the left and center of Figure~\ref{fig:pseudo}. (Pseudo-folds were introduced in \cite{HW}, and can also found in \cite{Hen} Definition 12.)  The length of the knot diagram is thus altered by removing a short length of the two line segments adjacent to the vertex, and replacing it with a short curved segment as illustrated on the right in Figure~\ref{fig:pseudo}. The curved segments also contribute to the error in the approximation of the embedded paper band to the folded ribbon knot, and this error is also proportional to the distance $\epsilon$. To see this, note that the length of the curved segment will depend on the fold angle at the vertex but can be bounded from above. Since a polygonal knot has a finite number of edges, say $N$, we note that two adjacent edges of the knot diagram have most $N-2$ other edges passing between them. Thus every error in the approximation is bounded above by a constant multiple of $\epsilon$. This means that, as $\epsilon$ goes to 0, the distance between any particular point on the folded ribbon knot and its approximating point on the smooth paper band goes to 0. In addition, these smooth paper bands have been constructed to as to respect the crossing and folding information of the folded ribbon knot. Thus their centerlines are of the same knot type as the knot type of $K$. In summary, we have given an outline of the proof of the following.

\begin{lemma}[Approximation] \label{lem:approx}
For every folded ribbon knot $K_w$, there is a sequence of smooth paper bands in $\R^3$ which converge pointwise to $K_w$, and whose aspect ratios converge to $\Rib(K_w)$. The centerline of $K_w$ and the center line of the paper bands have the same knot type. \qed
\end{lemma}

\section{Accordion Folds and the M\"obius band} \label{sect:mob}

We wish to construct a folded ribbon M\"obius band with as many half-twists as we like. To do this, we follow Hennessey's construction which uses the ``belt trick" (see \cite{Kauff-book}). Here a strip of paper is coiled around itself many times, then the ends are pulled apart without allowing any rotation. This turns the coils into half-twists, and is a physical interpretation of the well-known ``link, twist, writhe'' formula ($\Lk(K_w) = \Tw(K_w) + \Wr(K_w)$) for a knotted ribbon (see for instance \cite{Adams, Cal61, Cal59, Egg, Whi}).  A moment's experimentation reveals that the tricky part of this construction involves moving the end of the ribbon at the center of the coil so it may be joined to the end on the outside of the coil. Hennessey's solution~\cite{Hen} is to create an ``escape accordion" of fixed length as illustrated in Figure~\ref{fig:accord-crease}. In the notation and language of folded ribbon knots, this corresponds to creating an alternating sequence of left then right overfolds each of fold angle $\pi/2$.  If the escape accordion is unfolded, then the fold pattern on the ribbon is a sequence of mountain and valley folds of parallel lines at a $\pi/4$ angle to the edge of the ribbon (rightmost in Figure~\ref{fig:accord-crease}). 
\begin{center}
    \begin{figure}[htpb]
    \begin{overpic}{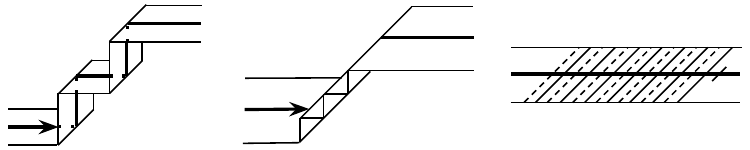}
    \end{overpic}
    \caption{On the left, a sequence of accordion folds with fold angle $\pi/2$. In the center, an escape accordion. On the right, the crease pattern for an escape accordion.}
    \label{fig:accord-crease}
    \end{figure}
\end{center}
As mentioned in Section~\ref{sect:intro}, Schwartz and Neinhaus conjecture that the infimal aspect ratio of a multi-twist paper M\"obius band is $3\sqrt{3}$.  In this paper we make a first attempt at this conjecture by showing that any folded ribbon M\"obius band can be constructed with folded ribbonlength $3\sqrt{3}+\epsilon$ for any $\epsilon>0$. To do this, we  construct a multi-twist folded ribbon M\"obius band using an escape accordion with an arbitrary fold-angle $\theta$, then prove that the folded ribbonlength is minimized when $\theta=\pi/3$. We then smoothly approximate this folded ribbon M\"obius band by a sequence of paper M\"obius bands in $\R^3$, giving an upper bound on the infimal aspect ratio for any multi-twist paper M\"obius band.

\subsection{Escape accordion}
For the rest of this paper we assume that our folded ribbon links have width $w=1$. 

\begin{definition}[Accordion fold] \label{def:folds} 
Take an oriented piece of folded ribbon. As shown on the left in Figure~\ref{fig:accord-crease}, an {\em accordion fold with fold angle $\theta$ and distance $d$} is an alternating sequence of left then right overfolds, each with fold angle $\theta$ and at equal distance to one another. Thus in an accordion fold, there is a constant distance $d$ between each vertex of the corresponding knot diagram. See Figure~\ref{fig:accord-math} left. Note that we could have chosen to construct accordion folds using only underfolds or by folding to the right at the start. {\bf In this paper we assume all accordion folds start with a left overfold.}  We denote the start and end vertices of the accordion fold by $v_S$ and $v_E$ respectively. We denote the distance between the start and end points of the accordion along the knot by $d_K(v_S, v_E)$. We denote the planar distance between the start and end points of the accordion fold by $d(v_S, v_E)$.
\end{definition}

We note that when the accordion fold is unfolded, the ribbon will show a sequence of parallel valley and mountain fold lines each at angle $\frac{\pi}{2}-\frac{\theta}{2}$ to the edge of the ribbon.  We also observe that after an even number of folds, the ribbon leaving the accordion fold is parallel to the ribbon entering the accordion fold. 

\begin{center}
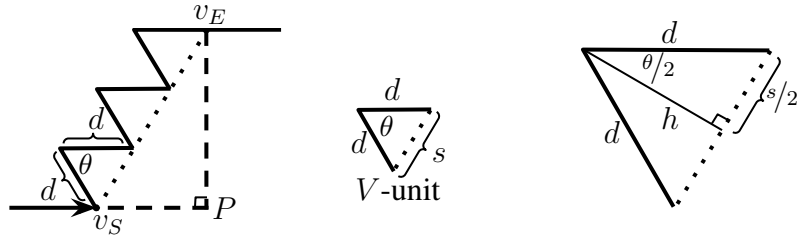
\begin{figure}[htpb]
\begin{overpic}{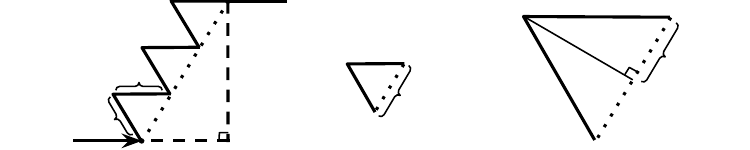}
\put(18.5,-1){$v_S$}
\put(29,21){$v_E$}
\put(31,0){$P$}
\put(13,2){$d$}
\put(17,4.75){\small{$\theta$}}
\put(18,9.5){$d$}
\put(49,12.5){$d$}
\put(45.5,7){$d$}
\put(48.5,8.75){\small{$\theta$}}
\put(54,6){$s$}
\put(46,2){$V$-unit}
\put(78,18.5){$d$}
\put(72, 8){$d$}
\put(76,15){${\nicefrac{\theta}{2}}$}
\put(89,12){$\nicefrac{s}{2}$}
\put(78,9.5){$h$}
\end{overpic}
\caption{On the left, the path of the knot diagram in an accordion fold. The center and right images show one $V$-unit. The dotted lines in these figures help us compute the distances connected to an accordion fold.}
\label{fig:accord-math}
\end{figure}
\end{center}

\begin{definition}[$V$-unit, escape accordion]\label{def:vunit}
In an accordion fold with fold angle $\theta$ and distance $d$, a {\em $V$-unit} is shown in the center and right images in Figure~\ref{fig:accord-math}. More precisely, a $V$-unit is the part of the knot diagram in an accordion fold containing three consecutive vertices corresponding to a left overfold, a right overfold, and a left overfold. At the end of a $V$-unit, the ribbon is parallel to its position entering the $V$-unit.   The knot diagram of a $V$-unit has length $2d$. An {\em escape accordion} is an accordion fold with an even number of folds, say $k=2m$ of them, such that the parallel pieces of ribbon entering and leaving the accordion fold are at least distance $w=1$ apart. \end{definition}

In Figure~\ref{fig:accord-math} left, we assume that line segment $v_EP$ is perpendicular to the line segment containing $v_SP$.  Thus for an {\em escape accordion}, we require the distance $d(v_E,P)$ to be at least the width $w=1$. Note that given a particular number of folds $k$, the distance $d$ can be adjusted so that $d(v_E,P)=1$.

\begin{lemma}\label{lem:vdist}
Assume there is an accordion fold with fold angle $\theta$ and distance $d$. The distance of a $V$-unit  along the knot is $2d$. The distance, $s$, between the start and end vertices of a $V$-unit is $s=2d\sin(\nicefrac{\theta}{2})$.  In an escape accordion with $k=2m$ folds, the distance between the start and end vertices of the escape accordion along the knot, $d_K(v_S, v_E)=kd$. The distance between these points in the plane is at least 
$$d(v_S, v_E)=kd\sin(\nicefrac{\theta}{2}) \geq \frac{1}{\cos(\nicefrac{\theta}{2})}.$$ Assuming equality gives $\ds kd=\frac{1}{\cos(\nicefrac{\theta}{2})\sin(\nicefrac{\theta}{2})}.$
\end{lemma}

\begin{proof} The distance $s$ and can be found using trigonometry from Figure~\ref{fig:accord-math} right. Then $d(v_S, v_E)=\frac{k}{2}s=kd\sin(\nicefrac{\theta}{2})$. In Figure~\ref{fig:accord-math} left, and note that angle $\angle v_Sv_EP=\frac{\theta}{2}$.  If we assume that distance $d(v_E,P)\geq 1$, the second identity follows. \end{proof}

\begin{center}
    \begin{figure}[htpb]
    \begin{overpic}{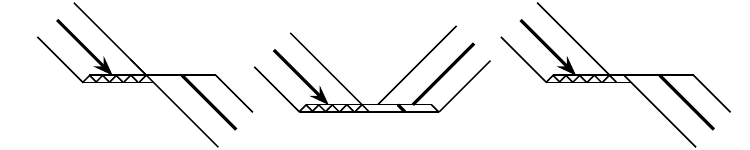}
    \put(20,1){}
    \put(60,10.5){}
    \put(37,1){}
    \put(87,7.5){}
    \end{overpic}
    \caption{On the left, an escape accordion. In the center, an escape accordion with one half-wrap. On the right, an escape accordion with two half-wraps.}
    \label{fig:crwrap}
    \end{figure}
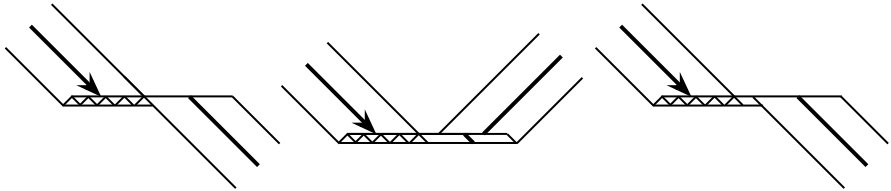
\end{center}

\subsection{M\"obius band}
We now give a detailed description of the construction of a multi-twist folded ribbon M\"obius band using an escape accordion (following Hennessey's ideas \cite{Hen}). Starting with an oriented piece of folded ribbon, we construct an escape accordion with fold angle $\theta$ and distance $d$. At this point we could continue the accordion fold by adding $V$-units, or we can wrap the ribbon around the escape accordion. We wrap the ribbon around the escape accordion by constructing a number of {\em half-wraps}: here, we fold an alternating sequence of a left underfold followed by a right overfold, each with the same fold angle $\theta$ and the same constant distance $d$ between each vertex of the corresponding knot diagram. Thus, in Figure~\ref{fig:crwrap}, the center image shows an escape accordion with one half-wrap; while the right image shows an escape accordion with two half-wraps.  Each half-wrap corresponds to a half-twist for the M\"obius band.
 We observe that the knot diagram for an accordion fold and a sequence of half-wraps looks the same, since the knot diagram does not show the folding information. 
 
\begin{center}
    \begin{figure}[htpb]
    \begin{overpic}{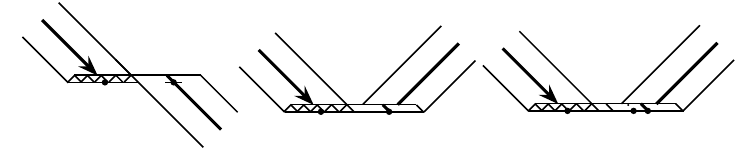}
    \put(13,6.75){$v_S$}
    \put(21.75,11){$v_E$}
    \put(42,2.75){$v_S$}
    \put(50,3){$v_E=w_1$}
    \put(74.5,3){$v_S$}
    \put(82.75,3){$w_1$}
    \put(86.5,3){$w_3$}
    \end{overpic}
    \caption{On the left, an escape accordion with labels. In the center, an escape accordion with one half-wrap with vertex $v_E=w_1$. On the right, an escape accordion with three labeled half-wraps}
    \label{fig:accord-notation}
    \end{figure}
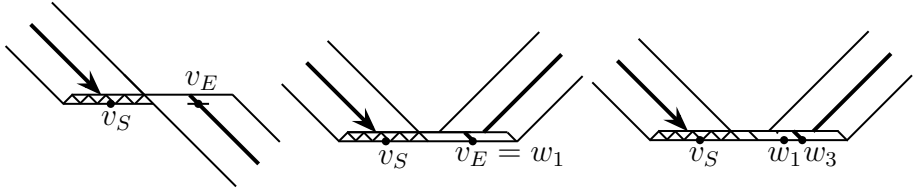
\end{center}

To help keep track of the distances, we label the half-wraps by the vertex of the knot diagram at the start of the half-wrap. Thus the end vertex of the escape accordion is the same vertex as the first wrap, or $v_E=w_1$. We add in labels $w_2, w_3, \dots, w_{2n+1}$ as we add in each half-wrap. Since the knot diagram is the same for a sequence half-wraps as it is for an accordion fold, we can compute the distances using the same arguments as Lemma~\ref{lem:vdist}. 

\begin{corollary}\label{cor:wrap-dist}
Assume there is an escape accordion, with fold angle $\theta$ and distance $d$,  that is followed with a sequence of $2n+1$ half-wraps labeled $w_1, w_2, \dots, w_{2n+1}$ as described above. Then the distance between vertices $w_1$ and $w_{2n+1}$ of the half-wraps along the knot, $d_K(w_1, w_{2n+1})=2nd$. The distance between these points in the plane is $d(w_1, w_{2n+1})=ns=2nd\sin(\nicefrac{\theta}{2})$.\end{corollary}

\begin{proof} By construction, the distance between any two consecutive vertices $d(w_k,w_{k+1})=d$ (for all $k\geq 1$). Recall that distance $s=2d\sin(\nicefrac{\theta}{2})$ is the distance between the start and end vertices of a $V$-unit. This is the same as the distances $d(w_{2k-1},w_{2k+1})$ or $d(w_{2k},w_{2k+2})$, for all $k\geq 1$.
\end{proof}

\begin{const}[M\"obius band] \label{const:Mobius}
We construct a multi-twist folded ribbon M\"obius band by taking an oriented piece of folded ribbon, then 
\begin{compactitem}
\item constructing an escape accordion with fold angle $\theta$ and distance $d$
\item adding $2n+1$ half-wraps, 
\item then folding the ends of the ribbon behind the escape accordion and joining. 
\end{compactitem} This is illustrated in Figure~\ref{fig:multi-mob}.  More precisely, if we orient the knot diagram from $A$ to $v_S$, then there is a left underfold at vertex $E$ and a left overfold at vertex $A$. Here, the line $AE$ is parallel to the folded edge $v_Sv_Ew_{2n+1}$. \qed
\end{const}

\begin{proposition}\label{prop:mtm-dist} Consider a multi-twist folded ribbon M\"obius band $\mathcal{M}$ constructed from an escape accordion with fold angle $\theta$ and distance $d$, followed by $2n+1$ half-wraps as described in Construction~\ref{const:Mobius}. This has folded ribbonlength
$$\Rib(\mathcal{M}) = \frac{2}{\cos(\nicefrac{\theta}{2})} + \frac{1}{\cos(\nicefrac{\theta}{2})\sin(\nicefrac{\theta}{2})} + \tan(\nicefrac{\theta}{2})  + 2nd(1+\sin(\nicefrac{\theta}{2})).
$$
\end{proposition}

\begin{center}
    \begin{figure}[htpb]
    \begin{overpic}{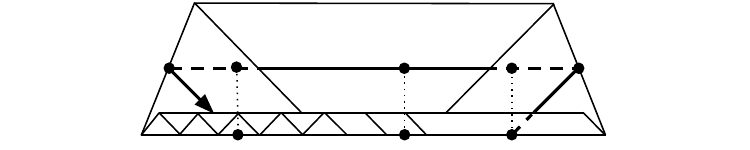}
    \put(19.,11){$A$}
    \put(31,-1){$v_S$}
    \put(28.75,11.75){$B$}
    \put(52.5,12){$C$}
    \put(50,-1){$v_E=w_1$}
    \put(69,12){$D$}    
    \put(67,-1){$w_{2n+1}$}    
    \put(78,11.5){$E$}
    \end{overpic}
    \caption{Construction of a M\"obius band with $2n+1$ half-twists. The labeled points help determine the ribbonlength. }
    \label{fig:multi-mob}
    \end{figure}
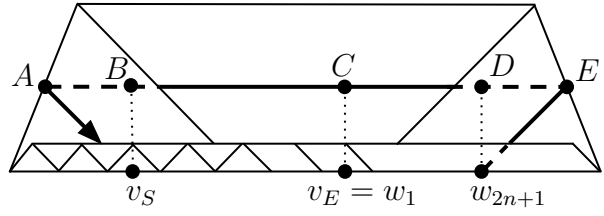
\end{center}

\begin{proof} We have assumed the width of the folded ribbon $w=1$. Thus the folded ribbonlength is just the length of the knot diagram.  In Figure~\ref{fig:multi-mob}, note that the line segments $Bv_S$, $Cv_E$ and $Dw_{2n+1}$ are perpendicular to line segments $AE$ and $v_Sw_{2n+1}$. Thus $d(B,C)=d(v_S,v_E)$ and $d(C,D) = d(w_1, w_{2n+1})$.
 We also observe that triangles $\triangle ABv_S$ and $\triangle EDw_{2n+1}$ are congruent. In $\triangle ABv_S$, angle $\angle Bv_SA=\nicefrac{\theta}{2}$, and length $d(B,v_S)=\nicefrac{1}{2}$. Thus $d(A,B)=\frac{1}{2}\tan(\nicefrac{\theta}{2})$ and $d(A, v_S)=\frac{1}{2\cos(\nicefrac{\theta}{2})}$. 
We combine these with the distance formulas from Lemma~\ref{lem:vdist} and Corollary~\ref{cor:wrap-dist} to find
\begin{align*} \Rib(\mathcal{M}) & = d_K(v_S,v_E) + d_K(w_1,w_{2n+1}) + 2d(A,v_S)+ 2d(A,B)+d(B,C)+ d(C,D)
\\ &= \frac{1}{\cos(\nicefrac{\theta}{2})\sin(\nicefrac{\theta}{2})} + 2nd + \frac{1}{\cos(\nicefrac{\theta}{2})} + \tan(\nicefrac{\theta}{2}) +  \frac{1}{\cos(\nicefrac{\theta}{2})} + 2nd\sin(\nicefrac{\theta}{2})
\\ & =  \frac{2}{\cos(\nicefrac{\theta}{2})} + \frac{1}{\cos(\nicefrac{\theta}{2})\sin(\nicefrac{\theta}{2})} + \tan(\nicefrac{\theta}{2})  + 2nd(1+\sin(\nicefrac{\theta}{2})).
\end{align*}
We note that for any $n$, the distance $d$ can be made as small as we like so that the last term $2nd(1+\sin(\nicefrac{\theta}{2}))\rightarrow 0$.
\end{proof}

\begin{remark} It is straightforward to compute the ribbon linking number of these folded ribbon M\"obius bands using the shortcuts developed in by the first author and Troy Larsen in  \cite{Den-FRLU}.  These arguments were also used by Hennessey (see Section 3.2 in \cite{Hen}) to observe that the ribbon linking number cancels in pairs of folds in the escape accordion, while each half-wrap increases the ribbon linking number.  Thus Construction~\ref{const:Mobius} can be used to construct a multi-twist folded ribbon M\"obius band with any odd ribbon linking number.
\end{remark}

\begin{theorem} \label{thm:min-angle}
Any multi-twist folded ribbon M\"obius band can be folded so that its folded ribbonlength is $3\sqrt{3}+\epsilon$ for any $\epsilon>0$. This M\"obius band is constructed from an escape accordion followed by half-wraps each with fold angle $\nicefrac{\pi}{3}$. 
\end{theorem}

\begin{proof} Construct a multi-twist folded ribbon M\"obius band $\mathcal{M}$ from an escape accordion followed by $2n+1$ half-wraps as described above. Proposition~\ref{prop:mtm-dist} gives the folded ribbonlength of this construction. We differentiate this folded ribbonlength with respect to $\theta$, and ignore the term involving distance $d$, as this may be made as small as we like.
\begin{align*} \frac{d\Rib(\mathcal{M})}{d\theta} & = \tan(\nicefrac{\theta}{2})\sec(\nicefrac{\theta}{2}) + \frac{1}{2}\sec^2(\nicefrac{\theta}{2}) + \frac{1}{2}\left(\sec^2(\nicefrac{\theta}{2}) -\csc^2(\nicefrac{\theta}{2})\right)
\\ & = \frac{2\sin^3(\nicefrac{\theta}{2}) + 2\sin^2(\nicefrac{\theta}{2}) - \cos^2(\nicefrac{\theta}{2})}{2\cos^2(\nicefrac{\theta}{2})\sin^2(\nicefrac{\theta}{2})}
\\ & = \frac{(\sin(\nicefrac{\theta}{2})+1)^2(2\sin(\nicefrac{\theta}{2})-1)}{2\cos^2(\nicefrac{\theta}{2})\sin^2(\nicefrac{\theta}{2})}
\end{align*}
When $ \frac{d\Rib(\mathcal{M})}{d\theta}=0$, we find the solution $\theta =\nicefrac{\pi}{3}$ in our range $0<\theta<\pi$. Another computation reveals that this critical point is a minimum of $\Rib(\mathcal{M})$.
 Finally, when $\theta =\nicefrac{\pi}{3}$, we compute that 
\begin{align*}\Rib(\mathcal{M}) &= \frac{2}{\cos(\nicefrac{\pi}{6})} + \frac{1}{\cos(\nicefrac{\pi}{6})\sin(\nicefrac{\pi}{6})} + \tan(\nicefrac{\pi}{6}) + 2nd(1+\sin(\nicefrac{\pi}{6}))
\\ & = \frac{4}{\sqrt{3}} + \frac{4}{\sqrt{3}} + \frac{1}{\sqrt{3}} + 2nd(1+\frac{1}{2})  = 3\sqrt{3}+3nd \leq 5.197.
\end{align*}
Since distance $d$ can be made as small as we like the last inequality and theorem follow.
\end{proof}

\begin{corollary} \label{cor:mmmb}
The infimal aspect ratio of a multi-twist paper M\"obius band is $\leq 3\sqrt{3}$.
\end{corollary}
\begin{proof} Any folded ribbon M\"obius band can be well approximated by a sequence of paper M\"obius bands so that the aspect ratios converge to the folded ribbonlength. This is proved in Lemma 3 in \cite{Hen}, or in Lemma~\ref{lem:approx} above.
\end{proof} 

How does this compare to previous work? We note that $3\sqrt{3}\leq 5.197$, which is less than the value of $3+2\sqrt{2}\leq 5.83$ 
given by Hennessey in \cite{Hen}.   Recall that Schwartz \cite{RES-Mob} proved that the infimal aspect ratio of an embedded paper M\"obius band is $\sqrt{3}$. In addition, Schwartz and Brown \cite{RES-Brown} proved that it is possible to build a 3 half-twist paper M\"obius band with aspect ratio 3. This is is conjectured infimal aspect ratio for this case. For the 5 half-twist paper M\"obius band case, we prove the following in Appendix~\ref{sect:6twist}.

\begin{proposition}\label{prop:5tw-mb} The infimal aspect ratio of a 5 half-twist paper M\"obius band is $\leq 5$. 
\end{proposition}

These results, together with Corollary~\ref{cor:mmmb}, lead us to conjecture the following.

\begin{conjecture}\label{conj:Mob} The infimal aspect ratio of an embedded paper M\"obius band with 7 or more half-twists is $3\sqrt{3}\leq 5.197$.
\end{conjecture}

The reader might wonder about the case for multi-twist paper annuli? In \cite{Hen}, Hennesey adjusts Construction~\ref{const:Mobius} to construct a multi-twist folded ribbon annulus by ending the construction with an even number of half-wraps. The ends of the ribbon are still folded behind and then joined. Hennessey computes the aspect of this multi-twist paper annulus and proves the following\footnote{Hennessey proves there is a multi-twist annulus with aspect ratio $\sqrt{2\alpha^2+2\sqrt{2}\alpha+ 2} + \sqrt{2}\alpha + 2 + \epsilon$, where $\alpha$ is the larger root of $(1+7\sqrt{2})x^2 -12x+2\sqrt{2}$ and $\epsilon >0$.  }.

\begin{theorem} [\cite{Hen}]\label{thm:annulus}
The infimal aspect ratio of a multi-twist paper annulus is $\leq 5. 38$.
\end{theorem}

Recall that recently Schwartz and Montgomery \cite{RES-Mont} have proved that the 2 half-twist paper annulus has aspect ratio at least 2.  In 2023, the first author and Larsen \cite{Den-FRLU} proved that any folded ribbon unknot $U_w(n)$ with ribbon linking number $\Lk(U_w(n))=\pm n$ can be constructed with folded ribbonlength $\Rib(U_w(n))=2n$. At the time, we conjectured that this was the infimal folded ribbonlength. This conjecture thus has been proved true for the $n=1$ case (or 2 half-twists) by Schwartz and Montgomery~\cite{RES-Mont}.  Theorem~\ref{thm:annulus} shows the conjecture is false for all ribbon linking numbers $n\geq 4$.  For the ribbon linking number $n=2$  and $n=3$ cases, we can build folded ribbon unknots with folded ribbonlength~$4$. These constructions are found in Appendix~\ref{sect:6twist}. This means that we have proved the following.

\begin{proposition}\label{prop:4-6-an} The infimal aspect ratio of both a 4 and 6 half-twist paper annulus is $\leq 4$. 
\end{proposition}

Finally, we expect that Theorem~\ref{thm:annulus} can be improved by using an escape accordion with fold angle $\theta=\pi/3$ (as in Construction~\ref{const:Mobius}), but we have not included this computation here. Putting all of this together, we expect that a conjecture similar to Conjecture~\ref{conj:Mob} holds.

\begin{conjecture}\label{conj:annulus} There is a positive constant $C$,  with $C<5.38$, such that the infimal aspect ratio of an embedded paper M\"obius band with 8 or more half-twists is $\leq C$.
\end{conjecture} 

We end this section with the observation that the folding patterns required for $n$ half-twist paper M\"obius bands and annuli when $n$ is large, are different from those when $n$ is small. We will see this observation repeated for  folded ribbon torus and twist knots in Sections~\ref{sect:torus} and~\ref{sect:twist}.

\section{Folded ribbonlength for $(2,q)$-torus knots} \label{sect:torus}
Recall that a torus link is a link which can be embedded on a torus. These are often denoted by $T(p,q)$, with $p$ corresponding to the number of times the link winds around the longitude of the torus, and $q$ corresponding to the number of times the link winds around the meridian. Figure~\ref{fig:torus} shows the $T(5,2)$ and  $T(2,2)$ torus links. There are several well-known facts about torus links that can be found in many introductory knot theory textbooks (for instance \cite{Adams, Crom, JohnHen}). For example, $T(p,q)$ torus links are equivalent to $T(q,p)$ torus links. The greatest common divisor of $p$ and $q$ corresponds to the number of components in the link, thus $T(p,q)$ is a knot if and only if $\gcd(p,q)=1$. 
 The crossing number of a $(p,q)$ torus link is $\Cr(T(p,q))=\min\{((p-1)\cdot q),(p \cdot (q-1))\}$. When applied to $(2,q)$-torus links, we see that the crossing number is $\Cr(T(2,q))= q$. 
\begin{center}
\begin{figure}[htbp]
\includegraphics[height=.8in]{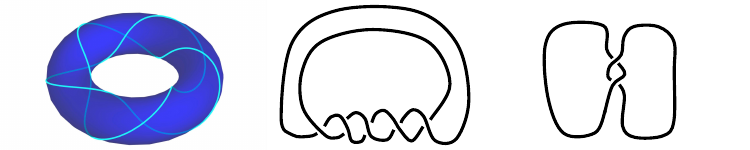}
\caption{On the left, the $T(2,5)$ torus knot on the torus. In the center, the $T(2,5)$ knot is arranged to show its five half-twists. On the right, the $T(2,2)$ torus link with two half-twists.}
\label{fig:torus}
\end{figure}
\end{center}

We now restrict our attention to $(2,q)$-torus knots (so $q$ is odd). We take the $q$ half-twists and arrange them as shown in Figure~\ref{fig:accord-layer} left. The torus knot is then created when ends $B$ and $C$ are joined and when ends $A$ and $D$ are joined.

\begin{center}
    \begin{figure}[htpb]
    \begin{overpic}{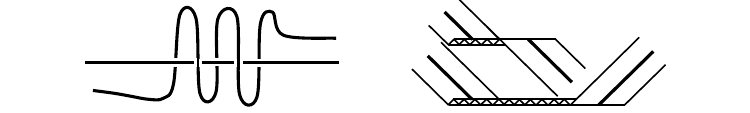}
    \put(10,0){$A$}
    \put(43,11){$B$}
    \put(10,8){$C$}
    \put(43,3){$D$}
    \put(56.5,13.5){$A$}
    \put(77.5,7){$B$}
    \put(54.5,7.5){$C$}
    \put(87.5,8.5){$D$}
    \end{overpic}
    \caption{On the left, we see 5 half-twists arranged so that a $T(5,2)$ torus knot is formed when ends $B$ and $C$ are joined and when ends $A$ and $D$ are joined. On the right, two accordion fold constructions labeled $AB$ and $CD$.}
    \label{fig:accord-layer}
    \end{figure}
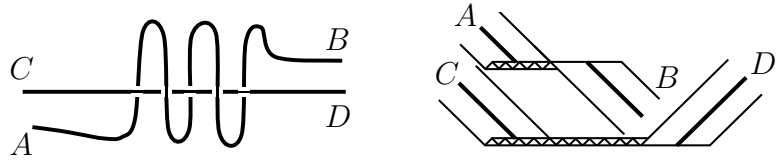
\end{center}

We translate these ideas into a new folded ribbon construction by using and altering Construction~\ref{const:Mobius}. We will also make use of Theorem~\ref{thm:min-angle} that a fold angle of $\nicefrac{\pi}{3}$ gives the minimum folded ribbonlength in the construction.  

\begin{center}
    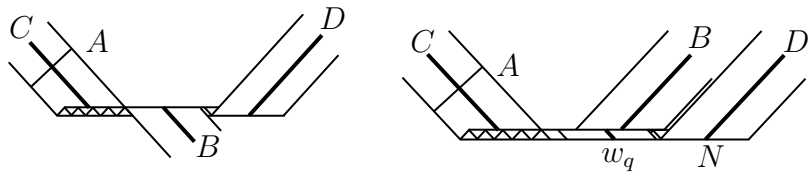
\begin{figure}[htpb]
    \begin{overpic}{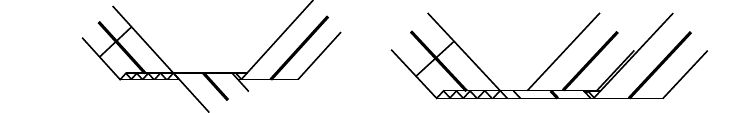}
    \put(19 ,11){$A$}
    \put(30.5,0){$B$}
    \put(11,12.5){$C$}
    \put(43.5,13.5){$D$}
    \put(62,8.5){$A$}
    \put(82,11.5){$B$}
    \put(53,11){$C$}
    \put(92,11){$D$}
    \put(73, -0.5){$w_q$}
    \put(83,-1){$N$}
    \end{overpic}
    \caption{On the left, escape accordion $AB$ is placed on top of $CD$. On the right, the end $B$ is wrapped around the thin accordion fold of $CD$ creating $q$ half-twists shown on the left in Figure~\ref{fig:accord-layer}.}
    \label{fig:accord-wrap}
    \end{figure}
\end{center}

\begin{const}[$(2,q)$-torus knots] \label{const:torus}
As shown in Figure~\ref{fig:accord-layer} right, first construct two escape accordions with fold angle $\nicefrac{\pi}{3}$ and distance $d$. We label the first by $CD$, and continue to add $V$-units so that it is a long accordion fold.  We label the second escape accordion by $AB$. We place escape accordion $AB$ over the long accordion fold $CD$ as shown in Figure~\ref{fig:accord-wrap} left. We then construct $q$ half-twists by taking end $B$ and folding $q$ half-wraps around the thin accordion fold of piece $CD$, as shown in Figure~\ref{fig:accord-wrap} right. We have marked the vertex at the start of the $q$-th wrap of end $B$ by $w_q$. (This is the same notation as in Construction~\ref{const:Mobius}.) Vertex $N$ marks start of the final accordion fold for end~$D$. The layering of the two pieces of ribbon means that the final fold of end $D$ takes place after the final fold of end $B$, as illustrated in Figure~\ref{fig:accord-wrap}. The additional ribbonlength needed to extend end $D$ past end $B$ is identical to the initial escape accordion. Thus $d(w_q,N) = d(v_S,v_E)$. 

We complete the construction of the $(2,q)$-torus knot by joining the appropriate ends. This is illustrated in Figure~\ref{fig:torus-join}.  In a similar way to Construction~\ref{const:Mobius}, we first fold the ends $C$ and $B$ behind the escape accordions and half-wraps and join. We then repeat with ends $A$ and $D$.  \qed
\end{const}

\begin{center}
    \begin{figure}[htpb]
    \begin{overpic}{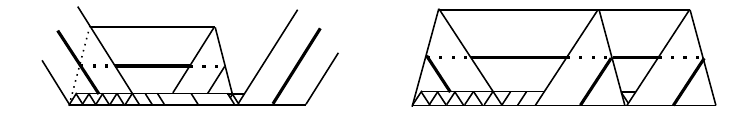}
    \put(5,11){$A$}
    \put(42,12){$D$}
    \put(18,12){$CB$}
    \put(67,14.5){$CB$}
    \put(83,14.5){$AD$}
    \end{overpic}
    \caption{Completing the construction of the $(2,q)$-torus knot. On the left ends $C$ and $B$ are joined, then on the right ends $A$ and $D$ are joined. }
    \label{fig:torus-join}
    \end{figure}
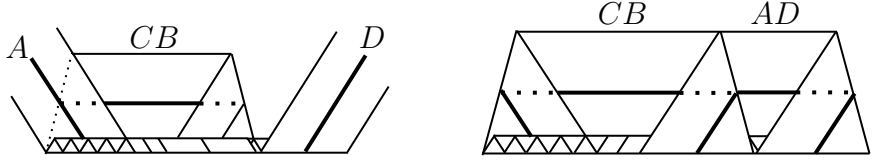
\end{center}

\begin{theorem} \label{thm:torus-bound}
Any $(2,q)$-torus knot type $K$ (where $q\geq 3$ is odd) contains a folded ribbon knot $K_w$ such that the folded ribbonlength is $\Rib(K_w)=8\sqrt{3}+\epsilon$ for any $\epsilon>0$. 
\end{theorem}

\begin{proof} We assumed the width of the folded ribbon $w=1$,  so we just need to compute the length of the knot diagram in Construction~\ref{const:torus} to find the folded ribbonlength. Let us also assume that $q=2n+1$.  To help us keep track of the distances, we  use the labels found in Figure~\ref{fig:torus-dist} (which is not to scale). We note the line segments $Fv_S$, $G{w_1}$, $Hw_{2n+1}$, and $JN$ are perpendicular to the line segments $EM$ and $v_SN$. The computations in Theorem~\ref{thm:min-angle} show us that
\begin{itemize} \item $d(E,v_S) = \frac{1}{2\cos(\nicefrac{\pi}{6})} = \frac{1}{\sqrt{3}}$,  
\item $d_K(v_S,v_E) =\frac{1}{\cos(\nicefrac{\pi}{6})\sin(\nicefrac{\pi}{6})} = \frac{4}{\sqrt{3}}$, 
\item $d_K(w_1,w_{2n+1})=2nd$, 
\item $d(E,F)= \frac{1}{2}\tan(\nicefrac{\pi}{6})=\frac{1}{2\sqrt{3}}$,  
\item $d(v_S,v_E) = \frac{1}{\cos(\nicefrac{\pi}{6})}=\frac{2}{\sqrt{3}}$, and
\item $d(w_1,w_{2n+1}) = 2nd\sin(\nicefrac{\pi}{6}) = nd$.
\end{itemize}

\begin{center}
    \begin{figure}[htpb]
     \begin{overpic}{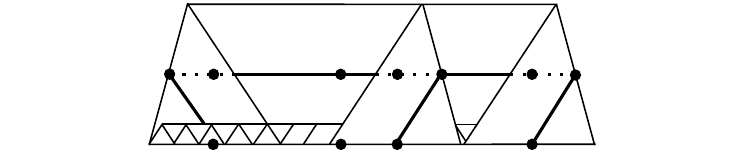}
    \put(19,10){$E$}
    \put(26,-1.5){\small{$v_S$}}
    \put(26,11){$F$}
    \put(44,11.5){$G$}
    \put(40,-1.5){\footnotesize{$v_E=w_1$}}
    \put(51,-1.5){\footnotesize{$w_{2n+1}$}}
    \put(52,11.25){\small{$H$}}    
    \put(59,11.5){$I$}
    \put(70,11.5){$J$}  
    \put(77.5,10){$M$}
    \put(69,-2.25){\footnotesize{$N$}}
    \put(39,20){$CB$}
    \put(62,20){$AD$}
    \end{overpic}
    \caption{The folded ribbon $(2,q)$-torus knot from Construction~\ref{const:torus}. The labeled points help determine the ribbonlength in Theorem~\ref{thm:torus-bound}. }
    \label{fig:torus-dist}
    \end{figure}
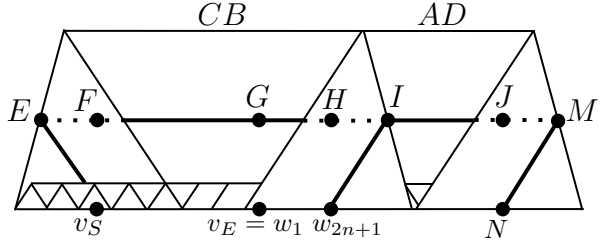
\end{center}

We compute the ribbonlength of the two pieces of ribbon labeled $AB$ and $CD$ used in Construction~\ref{const:torus}.  We measure each piece of ribbon starting vertex $E$ in Figure~\ref{fig:torus-dist}.  Starting with ribbon $AB$, we observe that end $B$ follows the knot diagram from vertex $E$ to vertex $v_s$, then travels in a zig-zag fashion through the escape accordion and wraps to vertex $w_{2n+1}$. We leave end $B$ at vertex $I$, where it will eventually join to end $C$. The distance traveled by end $B$ is  $2d(E,v_s)+ d_K(v_S, v_E) + d_K(w_1, w_{2n+1})$.
End $A$ follows the knot diagram from vertex $E$ to vertex $M$ (where it will eventually join to end $D$). Using the symmetry in the diagram, we find the distance traveled is \begin{align*} d(E,M) &=d(E,F)+d(F,G) + d(G,H)+d(H,J) +d(J,M)
\\ & =  2d(E,F) +2 d(v_S, v_E) + d(w_1, w_{2n+1}).
\end{align*}

We next compute the ribbonlength of the piece of ribbon labeled $CD$ in Construction~\ref{const:torus}.  End $D$ follows the knot diagram from vertex $E$ to to vertex $v_s$, then travels in a zig-zag fashion through the accordion folds to vertex $w_{2n+1}$, and then vertex $N$. We leave end $D$ at vertex $M$, where it joins with end $A$ (from above). The distance traveled is  $2d(E,v_s)+ 2d_K(v_S, v_E) + d_K(w_1, w_{2n+1})$.
Meanwhile, end $C$ follows the knot diagram from vertex $E$ to vertex $I$ where it joins with end $B$ (from above). The distance traveled is $2d(E,F) + d(v_S, v_E) + d(w_1, w_{2n+1})$.

We put these distances together to find
\begin{align*}
\Rib(K_w) &= 4d(E,v_S) + 3d_K(v_S,v_E) + 2d_K(w_1,w_{2n+1})+ 4d(E,F)\\ &\quad \quad \quad + 3d(v_S,v_E)+2d(w_1,w_{2n+1})
\\ & =\frac{4}{\sqrt{3}}+  \frac{12}{\sqrt{3}} + 4nd + \frac{2}{\sqrt{3}} + \frac{6}{\sqrt{3}} + 2nd
\\ & = 8\sqrt{3}+6nd \leq 13.86.  
\end{align*}
The last inequality and the theorem follow as the distance $d$ can be made as small as we like.
\end{proof}

At first glance, Theorem~\ref{thm:torus-bound} is surprising. The crossing number of any $(2,q)$-torus link is $\Cr(T(2,q))=q$, so we might have expected the folded ribbonlength to increase with crossing number.  However because we can make the distance $d$ between accordion folds as small as we like, we end up with a uniform upper bound. As an immediate application, we have solved the lower bound for ribbonlength crossing number problem from Equation~\ref{eq:crossing}.

\begin{theorem}\label{thm:main-result}
Given any knot $K$, the equation $c_1\cdot \Cr(K)^\alpha \leq \Rib([K])$ must have $\alpha=0$.
\end{theorem}

\begin{proof} The crossing number of any $(2,q)$-torus knot is $\Cr(T(2,q))=q$. However, Theorem~\ref{thm:torus-bound} shows that for all $q$ there is a folded ribbon $(2,q)$-torus knot $K_w$ such that $\Rib(K_w)\leq 13.86$.
\end{proof}

\begin{remark} We note that we can do a very similar construction for $(2,q)$-torus links (so $q$ is even). This involves folding $2n$ half-wraps, and then joining the ends appropriately.  We thus expect a similar result to Theorem~\ref{thm:torus-bound} to hold for $(2,q)$-torus links (so $q\geq 2$ is even), but we have not included this computation here.
\end{remark}

We know from from Lemma~\ref{lem:approx}, any folded ribbon knot can be well approximated by embedded paper bands with the same knot type. We thus deduce the following.

\begin{corollary} The infimal aspect ratio of an embedded paper $(2,q)$-torus knot is $\leq 8\sqrt{3}$.
\end{corollary}

Finally, what about other results for folded ribbon $(2,q)$-torus links?  In 2025, together with Henry Chen and Kyle Patterson \cite{CDPP}, we improved the upper bound for folded ribbonlength for a $(2,q)$ ribbon link. We proved that we can construct a folded ribbon $(2,q)$ torus link with $\Rib(T(2,q)_w)=q+3$. This gives yet another construction showing a trefoil knot can be constructed with a folded ribbonlength of $6$. In addition, this bound beats Theorem~\ref{thm:torus-bound} for $q=3, 5,7, 9$. It is possible that the construction in \cite{CDPP} gives the infimal folded ribbonlength for these small crossing $(2,q)$-torus knots. We also expect the construction in \cite{CDPP} gives the infimal folded ribbonlength for small crossing $(2,q)$-torus links. Given all this information, we close this section with the following conjecture.

\begin{conjecture} \label{conj:torus-bound}
The infimal folded ribbonlength for a $(2,q)$-torus link (where $q\geq 11$) is 
$\Rib([T(2,q)]) =8\sqrt{3}$.
\end{conjecture}

\section{Folded Ribbonlength for Twist knots} \label{sect:twist}
Like as a $(2,q)$-torus knot, a {\em twist knot $T_n$}, is constructed from $n$ half-twists. The difference is that a clasp is added to the ends of the $n$ half-twists to create the twist knot. The orientation of the clasp is usually chosen so that the corresponding knot diagram is alternating, as shown on the left in Figure~\ref{fig:twist-knots}.  For simplicity, we will call $T_n$ an {\em $n$-twist knot}. We thus expect to be able to give a uniform upper bound on the folded ribbonlength of any twist knot using the same ideas as in Section~\ref{sect:torus}. We first focus on twist knots where there is an odd number ($2n+1$) of half-twists. On the right in Figure~\ref{fig:twist-knots} we have arranged the half-twists in the diagram and labeled them, ready to start the folded ribbon construction.

\begin{center}
    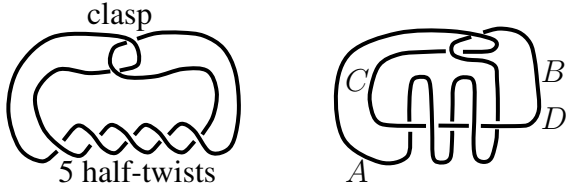
\begin{figure}[htpb]
    \begin{overpic}{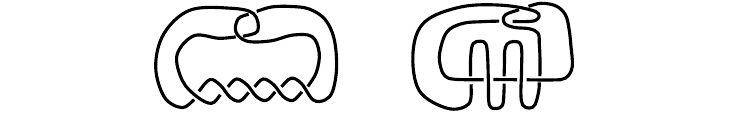}
    \put(29,15){clasp}
    \put(26, -1.5){5 half-twists}
    \put(56,-1.5){$A$}
    \put(76.5,9){$B$}
     \put(56,8){$C$}
     \put(76.5,4){$D$}
    \end{overpic}
    \caption{On the left, 5-twist knot in a standard diagram. On the right, the diagram shows the arrangement of half-twists used in Construction~\ref{const:twist}.}
    \label{fig:twist-knots}
    \end{figure}
\end{center}

\begin{const}[Folded ribbon $(2n+1)$-twist knots] \label{const:twist}
We start by using the same construction of half-twists as found in Construction~\ref{const:torus}. As shown in Figure~\ref{fig:accord-layer}, we start with two escape accordions with fold angle $\pi/3$. The first is labeled $CD$ and is made longer by adding additional $V$-units. The second is labeled $AB$ and is placed on top of $CD$ as in Figure~\ref{fig:accord-wrap}. We then add $2n+1$ half-wraps of end $B$ around $CD$ as shown on the left in Figure~\ref{fig:twist-start}. This creates $2n+1$ half-twists in the corresponding knot diagram. We now create a clasp in four steps.

{\bf Step 1:} Fold end $A$ with fold angle $\pi/3$ over to the right, so that end $A$ lies over the escape accordions and ends $B$ and $D$.  This is shown on the right in Figure~\ref{fig:twist-start}. The fold angle means that the lower edge of the ribbon near $A$ is parallel to the lower edge of the accordion folds.

\begin{center}
    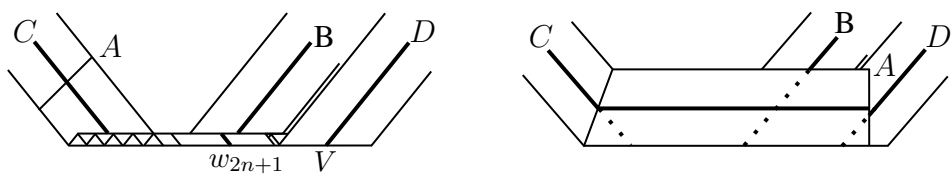
\begin{figure}[htpb]
    \begin{overpic}{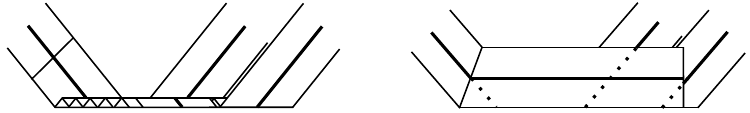}
    \put(10.5,10){$A$}
    \put(33,11){B}
    \put(1.5,11.5){$C$}
    \put(43,11.5){$D$}
    \put(22,-1.5){\small{$w_{2n+1}$}}
    \put(33,-2){\small{$V$}}
    \put(91.5,8){$A$}
    \put(87.5,12.5){B}
    \put(55.5,11){$C$}
    \put(97,11){$D$}
    \end{overpic}
    \caption{On the left,  the end $B$ is wrapped around the thin accordion fold of $CD$ creating the $(2n+1)$ half-twists. On the right, end $A$ is folded to the right over ends $B$ and $D$.}
    \label{fig:twist-start}
    \end{figure}
\end{center}

{\bf Step 2:} Fold end $B$ downward over end $A$ as shown on the left in Figure~\ref{fig:twist-clasp}. We arrange this fold so that the new fold line ends at the side-edge of the ribbon near end $A$. The ribbon near end $B$ is perpendicular to the ribbon near end $A$, and thus the local geometry means the fold angle here is $\pi/6$. 

{\bf Step 3:} Fold end $C$ with fold angle $\pi/3$ over to the right so that end $C$ lies over end $B$, as shown on the right in Figure~\ref{fig:twist-clasp}. Join the ends of $C$ and $A$ with a new fold line, and adjust so this fold line lies next to the side-edge of the ribbon near end $B$.

\begin{center}
    \begin{figure}[htpb]
    \begin{overpic}{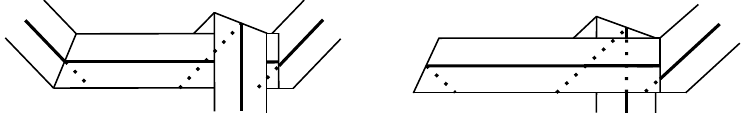}
    \put(37.5,9.5){$A$}
    \put(32.5,-1){$B$}
    \put(1.5,13){$C$}
    \put(43.5,12.5){$D$}
    \put(89,3){$A$}
    \put(84,-1){$B$}
    \put(88.5,8){$C$}
    \put(96.5,12){$D$}
    \end{overpic}
    \caption{On the left, end $B$ is folded down perpendicular to end $A$. On the right, end $C$ is folded to the right over end $B$, then joined to end $A$.}
    \label{fig:twist-clasp}
    \end{figure}
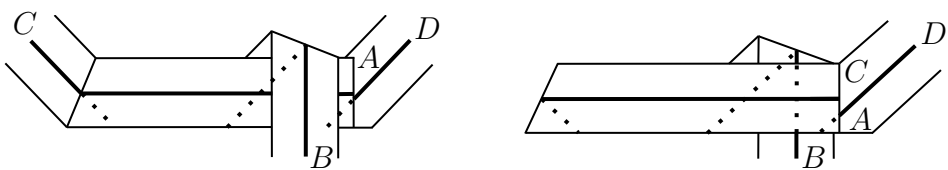
\end{center}

{\bf Step 4:} Fold end $D$ with fold angle $\pi/3$ over to the left so that it lies over end $CA$ as shown on the left in Figure~\ref{fig:twist-join}. We then fold end $D$ down with fold angle $\pi/2$ so that end~$D$ lies over end $B$ and is perpendicular to end $CA$ (shown on the right in Figure~\ref{fig:twist-join}). We then join ends $D$ and $B$ together with a new fold line so that this fold line lies next to the side-edge of end $CA$. The end result is shown in Figure~\ref{fig:final}. \qed
\end{const}

\begin{center}
    \begin{figure}[htpb]
    \begin{overpic}{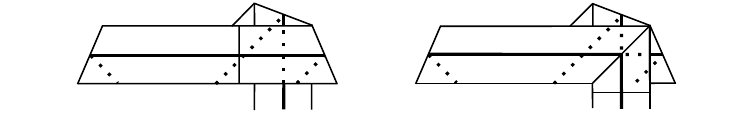}
    \put(42,11){$CA$}
    \put(38,-1){$B$}
    \put(28.5,8.5){$D$}
    \put(87,11.5){$CA$}
    \put(83,-1){$B$}
    \put(87,1){$D$}
    \end{overpic}
    \caption{On the left, end $D$ is folded to the left over end $CA$. On the right, end $D$ is folded down perpendicular to end $CA$ then joined to end $B$.}
    \label{fig:twist-join}
    \end{figure}
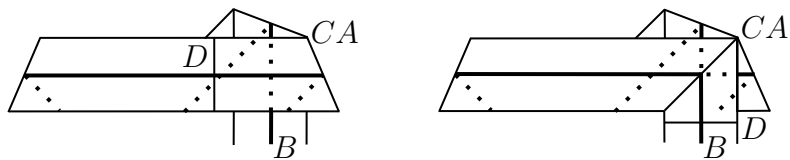
\end{center}

Our overall aim is to compute the folded ribbonlength of Construction~\ref{const:twist}. We have added labeled points to Figure~\ref{fig:final} to help with the computation. We have reused labels from Figure~\ref{fig:torus-dist}.  In order to easily understand Figure~\ref{fig:final}, we have enlarged two key areas on the right side (see Figure~\ref{fig:twist-compute}.

\begin{center}
    \begin{figure}[htpb]
    \begin{overpic}{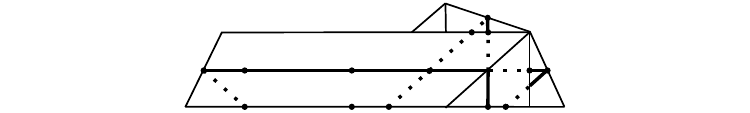}
    \put(24,5){$E$}
    \put(31,-1.5){$v_S$}
    \put(31,6.5){$F$}
    \put(46,6.5){$G$}
    \put(39, -1.5){$v_E=w_1$}
    \put(51,-1.5){$w_{2n+1}$}
    \put(56,6.5){$I$}
    \put(73.5,5){$J$}
    \put(66,-2){$V$}
    \put(63,-2){$U$}
    \put(68,6.25){\small{$T$}}
    \put(62,7.5){\small{$M$}}
    \put(64,13.5){$N$}
    \put(61,11){\small{$P$}}
    \put(70.5,11){$Q$}
    \end{overpic}
    \caption{The labeled points help compute distances in Construction~\ref{const:twist} of a $(2n+1)$ twist folded ribbon knot.}
    \label{fig:final}
    \end{figure}
\end{center}

\begin{center}
    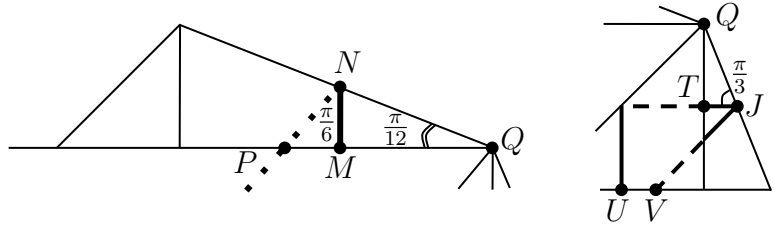
\begin{figure}[htpb]
    \begin{overpic}{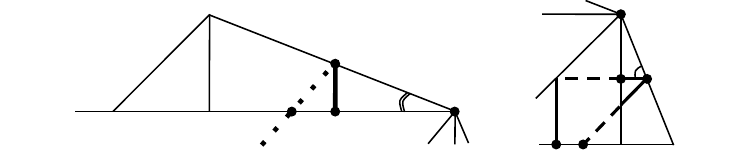}
    \put(33.5,2.5){$P$}
    \put(43,2){$M$}
    \put(44,13){$N$}
    \put(61.5,5){$Q$}
    \put(42.25,7){$\frac{\pi}{6}$}
    \put(49,6.5){\small{$\frac{\pi}{12}$}}
    \put(80,10.5){$T$}
    \put(72.5,-2.5){$U$}
    \put(76.5,-2.5){$V$}
    \put(87,8.5){$J$}
    \put(85.5,12){$\frac{\pi}{3}$}
    \put(84,18){$Q$}
    \end{overpic}
    \caption{Two enlarged parts of the  right side of Figure~\ref{fig:final}.}
    \label{fig:twist-compute}
    \end{figure}
\end{center}

\begin{lemma} \label{lem:tw-details}
Assume a folded ribbon $(2n+1)$-twist knot has been made following Construction~\ref{const:twist}. Given the corresponding diagram in Figure~\ref{fig:twist-compute}, the distance along the knot diagram from $M$ to $N$ to $P$ is $d_K(M,P)=\frac{1}{2\sqrt{3}}$, the planar distance $d(P,M)=\frac{1}{\sqrt{3}} -\frac{1}{2}$,  and the distance from $J$ to $T$ is $d(J,T)=\frac{1}{2\sqrt{3}}$.
\end{lemma}

\begin{proof}
To compute the distance $d_K(M,P)$, we look at diagram on the left in Figure~\ref{fig:twist-compute}. This shows the details in Step 2 of the clasp in Construction~\ref{const:twist}. Fold angle $\angle MNP=\frac{\pi}{6}$, thus the geometry of the fold at $N$ means angle $\angle MNQ = \frac{\pi - \pi/6}2 = \frac{5\pi}{12}$. From Step 2, we know that line segment $MN$ is perpendicular to $PQ$, thus angle $\angle NQM=\frac{\pi}{12}$. We have assumed the width of the folded ribbon $w=1$, thus side $MQ$ has length $\frac{1}{2}$. From triangle $\triangle MNQ$, we have $\frac{d(M,N)}{1/2}=\tan(\frac{\pi}{12}) = 2-\sqrt{3}$, hence $d(M,N)= 1-\frac{\sqrt{3}}{2}$. From triangle $\triangle MNP$, we have $\frac{d(M,N)}{d(N,P)}=\cos(\frac{\pi}{6})=\frac{\sqrt{3}}{2}$, hence $d(N,P)= \frac{2-\sqrt{3}}{\sqrt{3}} = \frac{2\sqrt{3}}{3} - 1$. Putting this altogether, we find 
$$d_K(M,P) = (1-\frac{\sqrt{3}}{2}) + (\frac{2\sqrt{3}}{3} - 1) = \frac{\sqrt{3}}{6}=\frac{1}{2\sqrt{3}}.$$
To find the planar distance $d(P,M)$, observe that $\frac{d(P,M)}{d(M,N)}=\tan(\frac{\pi}{6})=\frac{1}{\sqrt{3}}$. Therefore 
$$d(P,M)=\frac{1-\frac{\sqrt{3}}{2}}{\sqrt{3}} = \frac{1}{\sqrt{3}} -\frac{1}{2}.$$

To compute the distance  $d(J,T)$, we look at the diagram on the right in Figure~\ref{fig:twist-compute}. This shows the details in Step 4 of the clasp in Construction~\ref{const:twist}.  The path from $V$ to $J$ to $T$ along the knot is the path end $D$ takes as it folds to the left with fold angle $\frac{\pi}{3}$. Thus angle $\angle VJT = \frac{\pi}{3}$ and the geometry of the fold at $J$ means that angle $\angle QJT =\frac{\pi}{3}$ too. Distance $d(Q,T)=\frac{1}{2}$ as it is half of the ribbon width. Thus in right triangle $\triangle QTJ$, we have $d(J,T) = \frac{1/2}{\tan(\pi/3)} = \frac{1}{2\sqrt{3}}$.
\end{proof}

\begin{theorem}\label{thm:twist-bound}
Any folded ribbon $(2n+1)$-twist knot type $K$ contains a folded ribbon knot $K_w$ such that its folded ribbonlength is $\Rib(K_w)=9{\sqrt{3}} + 2 +\epsilon$ for any $\epsilon>0$. 
\end{theorem}

\begin{proof}  We compute the folded ribbonlength of the $(2n+1$)-twist knot from Construction~\ref{const:twist}. To do this, we just need to compute the length of the knot diagram, since  we assumed the width of the folded ribbon $w=1$. We will refer to the distances in Figure~\ref{fig:final-2} throughout this proof. Note that many of them were computed in Theorem~\ref{thm:torus-bound} and Lemma~\ref{lem:tw-details}.

\begin{center}
    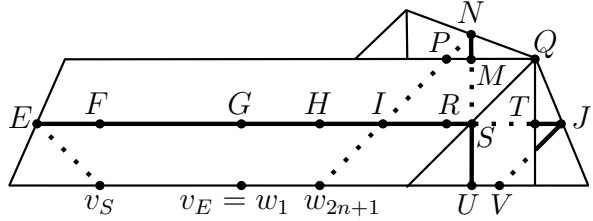
\begin{figure}[htpb]
    \begin{overpic}{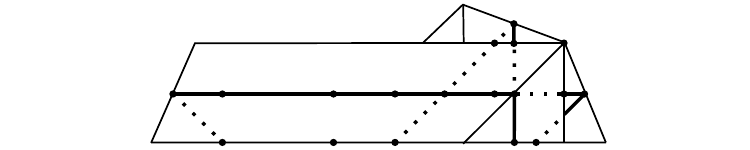}
    \put(20,7){$E$}
    \put(28,-1.5){$v_S$}
    \put(28,8.5){$F$}
    \put(43,8.5){$G$}
    \put(38, -1.5){$v_E=w_1$}
    \put(51,-1.5){$w_{2n+1}$}
    \put(51,8.5){$H$}
    \put(58,8,5){$I$}
    \put(79,7){$J$}
    \put(70.5,-2){\small{$V$}}
    \put(67,-2){\small{$U$}}
    \put(69,11.5){\small{$M$}}
    \put(67,18){$N$}
    \put(64,15){\small{$P$}}
    \put(75,15){$Q$}
    \put(65,8.5){\small{$R$}}
    \put(69,5){\small{$S$}}   
    \put(72.5,8.25){\small{$T$}}
    \end{overpic}
    \caption{Distances used in  the proof of Theorem~\ref{thm:twist-bound}.}
    \label{fig:final-2}
    \end{figure}
\end{center}

We compute the ribbonlength of the two pieces of ribbon labeled $AB$ and $CD$ used in Construction~\ref{const:twist}.  We measure each piece of ribbon starting vertex $E$ in Figure~\ref{fig:final-2}. We start with the piece of ribbon labeled $AB$.   End $B$ follows the knot diagram from vertex $E$ to vertex $v_S$, then travels in a zig-zag fashion through the accordion folds and wraps to vertex $w_{2n+1}$. This distance traveled is 
$d(E,v_S) + d_K(v_S,v_E) + d_K(w_1,w_{2n+1})$.  Following Step~2 of the clasp, we see end $B$ follows the knot diagram from vertex $w_{2n+1}$ to $P$ to $N$ to $M$, and ends at $U$ (where it will eventually join to end $D$). From Lemma~\ref{lem:tw-details}, we have $d_K(P,M)=\frac{1}{2\sqrt{3}}$. The other distances are $d(w_{2n+1}, P) = 2d(v_S,E)$, and $d(M,U)=1$ (as this is the width of the ribbon).  

Meanwhile, from Step 1, we see end $A$ starts at vertex $E$ and ends at vertex $T$ (where it will eventually join to end $C$). Observe that triangles $\triangle EFv_S$, $\triangle IHw_{2n+1}$ and $\triangle IRP$ are all congruent to one another, thus $d(R,S)=d(P,M)$, and $d(S,T)=\frac{1}{2}$ (as it is half the width of the ribbon). From Lemma~\ref{lem:tw-details}, we see $d(R,T) =(\frac{1}{\sqrt{3}} - \frac{1}{2}) + \frac{1}{2} =\frac{1}{\sqrt{3}}$.
We find that
\begin{align*} d(E,T) &= d(E,F)+d(F,G)+d(G,H)+d(H,I)+d(I,R)+d(R,T)
\\ & = 3d(E,F)+d(v_S,v_E) + d(w_1,w_{2n+1})+d(R,T).
\end{align*}
Let us now find the length of the piece of ribbon labeled $CD$ in Construction~\ref{const:twist}. Just as with end $A$, end $C$ starts at vertex $E$ and ends at vertex $T$, which adds a second distance $d(E,T)$.  In a similar way to end $B$, end $D$ follows the knot diagram from vertex $E$ to vertex $v_S$, then travels in a zig-zag fashion through the accordion folds to vertex $w_{2n+1}$, then through a second escape accordion to vertex $V$. The distance traveled is $d(E,v_S) + 2d_K(v_S,v_E) + d_K(w_1,w_{2n+1})$. From Step 4 of the clasp, end $D$ then follows the knot diagram from vertex $V$ to $J$, then $J$ to $T$ to $S$ to $U$. The distance traveled is $d(E,v_s) + d(J,T) + d_K(T,U)$, and recall from Remark~\ref{rmk:length} that $d_K(T,U)=1$. 

When we put all of this together we find that the folded ribbonlength of our $(2n+1)$-twist knot is
\begin{align*}
\Rib(K_w) & =  5d(E,v_S) + 3d_K(v_S,v_E) + 2d_K(w_1,w_{2n+1}) + 6d(E,F)  +2d(R,T)
\\ & \quad   + 2d(v_S,v_E) + 2d(w_1,w_{2n+1})+ d_K(P,M)+d(M,U) + d(J,T) +d_K(T,U)
\\ & = \frac{5}{\sqrt{3}} + \frac{12}{\sqrt{3}} + 4nd + \frac{3}{\sqrt{3}}+\frac{2}{\sqrt{3}} 
+  \frac{4}{\sqrt{3}} + 2nd+\frac{1}{2\sqrt{3}} + 1 +\frac{1}{2\sqrt{3}} + 1
\\ & = \frac{27}{\sqrt{3}} + 2 + 6nd =9\sqrt{3}+2+6nd  \leq 17.59. 
\end{align*}
The last inequality and the theorem follow as the distance $d$ can be made as small as we like.
\end{proof}

What about twist knots for an even number of half-twists? Both Construction~\ref{const:twist} and Theorem~\ref{thm:twist-bound} can be adjusted to account for this case.  Rather than repeat this here, the construction of a folded ribbon $2n$-twist knot (Construction~\ref{const:twist-even}) can be found in Appendix~\ref{sect:even-twist}.

\begin{theorem}\label{thm:twist-even}
Any folded ribbon $2n$-twist knot type $K$ contains a folded ribbon knot $K_w$ such that its folded ribbonlength is $\Rib(K_w)=8\sqrt{3} + 2 +\epsilon$ for any $\epsilon>0$.  
\end{theorem}

\begin{proof} Found in Appendix~\ref{sect:even-twist}.
\end{proof}

\begin{corollary}\label{cor:twist}
For any $n\geq 1$, the infimal folded ribbonlength of a twist knot $T_n$ is
$$\Rib([T_n])\leq\begin{cases} 8\sqrt{3} + 2 \leq 15.86 & \text{ when $n$ is even,}
\\ 9\sqrt{3}+2 \leq 17.59 & \text{ when $n$ is odd.}
\end{cases}
$$
\end{corollary}
\begin{proof} Choose $\epsilon >0$ small enough in Theorems~\ref{thm:twist-bound} and~\ref{thm:twist-even}.
\end{proof}

We know from from Lemma~\ref{lem:approx}, that any folded ribbon knot can be well approximated by embedded paper bands with the same knot type. We expect the infimal aspect ratio of an embedded paper twist knot  to be the same as in Corollary~\ref{cor:twist}.

How does Corollary~\ref{cor:twist} relate to other twist knots? Again, in \cite{CDPP}, we proved that any $n$-twist knot $K$ can be constructed with folded ribbonlength $\Rib(K_w)=n+6$. This means a figure-8 knot (a $T_2$ knot) can be constructed with folded ribbonlength $8$, the lowest found so far. This bound also beats Corollary~\ref{cor:twist} for $n=1, 2, 3,\dots, 9, 11$. This again illustrates the difference between the constructions for small crossing knots versus constructions showing universal upper bounds on folded ribbonlength.

We also observe that Corollary~\ref{cor:twist} gives us a second knot family that proves Theorem~\ref{thm:main-result} (which gives the lower bound on the folded ribbon crossing number problem).  Are there other infinite families of knots with uniform upper bounds on folded ribbonlength?  The pretzel link family $P(p,q,r)$ is a family of knots and links which are made of three sets of half-twists joined in a specific way. Construction~\ref{const:torus} shows us that any number of half-twists can be constructed with a finite amount of folded ribbonlength.  There is also a finite amount of ribbonlength needed to connect the ends of the three strands of half-twists in order to construct a pretzel link. This leads us to conjecture the following.

\begin{conjecture} There is a constant $C>0$, such that for all folded ribbon $P(p,q,r)$-pretzel links, the infimal folded ribbonlength $\Rib([P(p,q,r)])\leq C$.
\end{conjecture}

Just as before, we don't expect the constant $C$ in the conjecture to be the best upper bound for pretzel links with small crossing number. Our prior joint work from \cite{CDPP} shows the infimal folded ribbonlength of $P(p,q,r)$ pretzel links is bounded as follows:
 $$\Rib([P(p,q,r)]) \leq |p|+|q|+|r|+6.$$

We do not expect the methods used for $(2,q)$-torus links, twist knots, or pretzel links to generalize to other torus knots as the twists involve more than two strands. Indeed, in~\cite{Den-FRF}, the first author and co-authors proved that any $(p,q)$ torus link type $L$ (with $p\geq q\geq 2)$ contains contains a folded ribbon link $L_w$ such that $\Rib(L_w)\le 2p$. When we bring in crossing number, we get $\Rib([L])\leq c(\Cr(L))^{1/2}$ for a constant $c$ (see~\cite{Den-FRF}). However, is it also true that $c(\Cr(L))^{1/2}\leq\Rib([L])$ in this case? What about other knot and link families? We end with the following open question.

\begin{open} Is there an infinite family of knots for which $c\cdot \Cr(K)^\alpha \leq \Rib(K)$ must have $\alpha>0$ for some constant $c$? If there is such a family,  can we find one where $\alpha=\frac{1}{2}$?
\end{open}


\section{Acknowledgments}
We would like to thank the generous support of Washington \& Lee University (W\&L). Timi Patterson's research was funded by W\&L's 2024 Summer Research Scholars Program. Elizabeth Denne's research was funded by a 2024 Lenfest grant from W\&L.

A huge thanks go to Richard Schwartz and Aidan Hennessey. Aiden's generosity in explaining his methods inspired us to apply his ideas to folded ribbon knots.  Richard's encouragement and support was essential to the first author in finalizing the details of the multi-twist M\"obius band result.

All of the figures in this paper were created using Google Draw or {\em Mathematica}.

\appendix\section{Aspect ratios of paper bands with $\leq 6$ half-twists}
\label{sect:6twist}

In this section we review how to construct paper bands with $2, 3$ and $4$ half-twists with the lowest known aspect ratios. We then show how to construct paper bands with $5$ and $6$ half-twists with conjectured infimal aspect ratios. In all of these cases, we construct folded ribbon unknots with the appropriate number of half-twists. These folded ribbon unknots can the be approximated as closely as we like by the desired embedded paper bands.

\begin{center}
    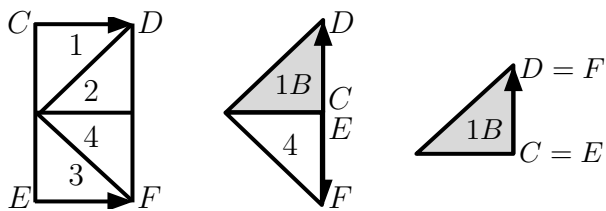
\begin{figure}[htpb]
    \begin{overpic}{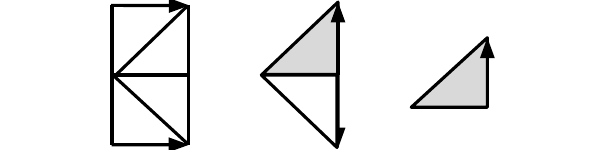}
    \put(23,20){$1$}
    \put(15,23){$C$}
    \put(32,23){$D$}
    \put(25,14){$2$}
    \put(23,3){$3$}
    \put(15,0){$E$}
    \put(32,0){$F$}
    \put(25,8){$4$}
    \put(50,15){\small{$1B$}}
    \put(57,13){$C$}
    \put(57,9){$E$}
    \put(57,23){$D$}
    \put(57,0){$F$}
    \put(51,7){$4$}
    \put(75,9){\small{$1B$}}
    \put(82,6){\small{$C=E$}}
    \put(82,17){\small{$D=F$}}
    \end{overpic}
    \caption{Construction of a 2 half-twist paper annulus with aspect ratio 2.}
    \label{fig:UnK-2HTw}
    \end{figure}
\end{center}

Recall that Schwartz and Montgomery in \cite{RES-Mont} proved that 2 half-twist paper annulus has aspect ratio $\geq 2$.   The same construction also appeared as a folded ribbon unknot with ribbon linking number 1 in \cite{Den-FRLU}. 
We show a construction similar to the one in \cite{RES-Mont}, as we will need this later on.

\begin{const}[2 half-twist paper annulus from \cite{RES-Mont}] \label{const:2-HTw-an} On the left in Figure~\ref{fig:UnK-2HTw}, we find a rectangle with aspect ratio~2, and an annulus is formed when sides $CD$ and $EF$ are identified as shown. The back side of the rectangle is shaded light grey. Two half twists are added as follows. Triangle 1 is folded over triangle 2 so edge $CD$ lies to the right, and triangle 3 is folded under triangle 4 so edge $EF$ lies to the right. This is shown in the center image, where $1B$ is the back of triangle 1. Finally, triangles 3 and 4 are folded under triangles 1 and 2 and edges $CD$ and $EF$ are identified as shown in the right image.  If you drill down in the faces in  the right image, you will encounter the triangles in the order 1, 2, 3, 4. That this construction does indeed have 2 half-twists, equivalently ribbon linking number~1,  was proved in \cite{Den-FRLU}. \qed
\end{const}

Two different  3 half-twist paper M\"obius bands both with aspect ratio 3 were constructed by Brown and Schwartz in \cite{RES-Brown}. Since we need it later,  we repeat Brown and Schwartz' construction of the version which lies in the plane, which we call the {\em flat model}. That the flat model corresponds to a folded ribbon unknot with ribbon linking number $3$ was proven in \cite{RES-Brown}. 
\begin{center}
    \begin{figure}[htpb]
    \begin{overpic}{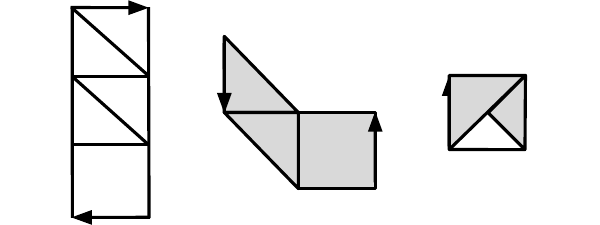}
    \put(16,27){$1$}
   \put(8,35){$C$}
  \put(25,35){$D$}
   \put(20,31){$2$}
    \put(16,15){$3$}
  \put(20,20){$4$}
  \put(17,6){$5$}
    \put(25,0){$E$}
    \put(8,0){$F$}
    \put(39,21.5){\small{$2B$}}
    \put(43,14){\small{$3B$}}
    \put(33,30){$C$}
  \put(33,18){$D$}
  \put(54,11){\small{$5B$}}
    \put(63,5){$E$}
    \put(63,18){$F$}
    \put(76,20){\small{$1B$}}
     \put(83,18){\tiny{$3B$}}
  \put(80,14){$5$}     
   \put(69,9){\small{$C=E$}}
  \put(69,26){\small{$D=F$}}
    \end{overpic}
    \caption{Construction of a 3 half-twist paper M\"obius band with aspect ratio 3.}
    \label{fig:UnK-3HTw}
    \end{figure}
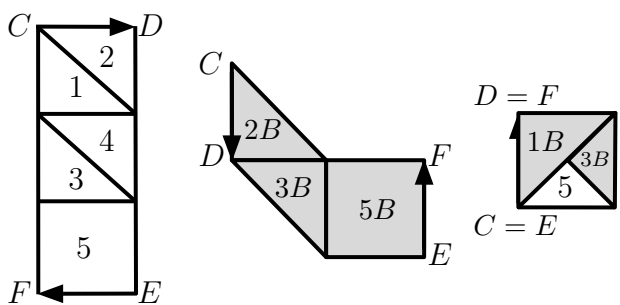
\end{center}

\begin{const}[Flat 3 half-twist paper M\"obius band from~\cite{RES-Brown}] \label{const:3-HTw-mb} A rectangle with aspect ratio 3 is shown on the left image in Figure~\ref{fig:UnK-3HTw}, and a M\"obius band is formed when we identify sides $CD$ and $EF$ as shown. The back side of the rectangle is shaded light grey. We add in 3 half-twists as follows. First fold triangle 2 over triangle 1 to the left so side $CD$ lies to the left and the back side of side 2, denoted $2B$ is uppermost. Then fold triangle 3 over triangle 4 to the right so the back sides of triangle 3 and square 5 are upper most and side $EF$ lies to the right. This is shown in the center image. To complete the construction, fold triangles 1 and 2 down over triangles 3 and 4. Then fold square 5 under and to the left (so the front side is uppermost) and identify edges $CD$ and $EF$. The right image shows the finished construction. If we drill down in the model in the center top part of the square, we will encounter faces 1, 2, 3, 4, 5 in that order.  \qed
\end{const}

\begin{center}
    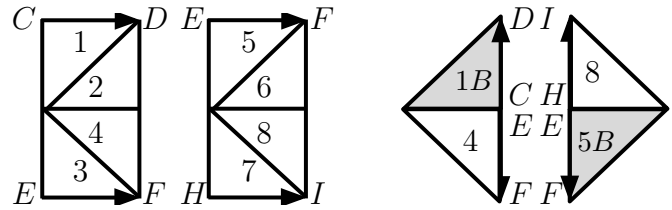
\begin{figure}[htpb]
    \begin{overpic}{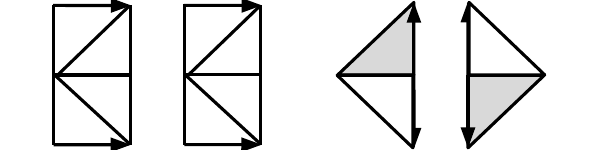}
    \put(13,20){$1$}
    \put(5,23){$C$}
    \put(22,23){$D$}
    \put(15,14){$2$}
    \put(13,3){$3$}
    \put(5,0){$E$}
    \put(22,0){$F$}
    \put(15,8){$4$}
    \put(35,20){$5$}
    \put(27,23){$E$}
    \put(44,23){$F$}
    \put(37,14){$6$}
    \put(35,3){$7$}
    \put(27,0){$H$}
    \put(44,0){$I$}
    \put(37,8){$8$}
    \put(70,13){$C$}
    \put(70,9){$E$}
    \put(63, 15){\small{$1B$}}
    \put(70,23){$D$}
    \put(64,7){$4$}
    \put(70,0){$F$}
    \put(74,9){$E$}
    \put(79,7){\small{$5B$}}
    \put(74,0){$F$}
    \put(74,13){$H$}
    \put(80, 16){$8$}
     \put(74,23){$I$}
    \end{overpic}
    \caption{Construction of a 4 half-twist paper annulus with aspect ratio 4.}
    \label{fig:UnK-4HTw}
    \end{figure}
\end{center}

\begin{const}[4 half-twist paper annulus from \cite{Den-FRLU}]\label{const:4-HTw-an}
In order to construct a 4 half-twist paper annulus with aspect ratio 4, simply join two 2 half-twist paper annuli from Construction~\ref{const:2-HTw-an} together. These paper annuli are shown  on the left in Figure~\ref{fig:UnK-2HTw}, where the  back side of each rectangle is shaded light grey.  Here, we join the two 2 half-twist paper annuli along side $EF$.  Begin by folding 2 half-twists in each model as in Construction~\ref{const:2-HTw-an}, so that triangle 1, respectively triangle 5, lies over triangle 2, resp. triangle 6. Also triangle 3, respectively triangle 7, lies under triangle 4, resp. triangle 8. The right image in Figure~\ref{fig:UnK-2HTw} shows how to complete the construction by joining the two side $EF$ sides, and then  by joining the two sides $CD$ and $HI$ as shown. \qed 
 \end{const}

%

\begin{center}
    \begin{figure}[htpb]
    \begin{overpic}{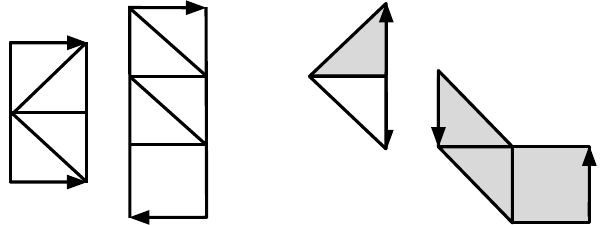}
    \put(5,25){$1$}
    \put(-2,29){$C$}
    \put(15,29){$D$}
    \put(9,21){$2$}
    \put(5,10){$3$}
    \put(-2,5){$E$}
    \put(15,5){$F$}
    \put(9,14){$4$}
   \put(18,35){$E$}
  \put(35,35){$F$}
   \put(26,27){$1$}
   \put(29,31){$2$}
    \put(26,15){$3$}
  \put(29,20){$4$}
  \put(27,6){$5$}
    \put(35,0){$H$}
    \put(18,0){$I$}
    \put(57,26){\small{$1B$}}
    \put(65,36){$D$}
    \put(65,25){$C$}
    \put(65,21){$E$}
    \put(65,11){$F$}
    \put(59,20){$4$}
   \put(75,15){\small{$2B$}}
    \put(79,8.5){\small{$3B$}}
    \put(69,23){$E$}
  \put(69,12){$F$}
  \put(89,5){\small{$5B$}}
    \put(99,-1){$H$}
    \put(99,12){$I$}
    \end{overpic}
    \caption{A 2 half-twist paper annulus with aspect ratio 2 is connected to a 3-half-twist paper M\"obius band with aspect ratio 3 along edge $EF$.}
    \label{fig:UnK-5HTw}
    \end{figure}
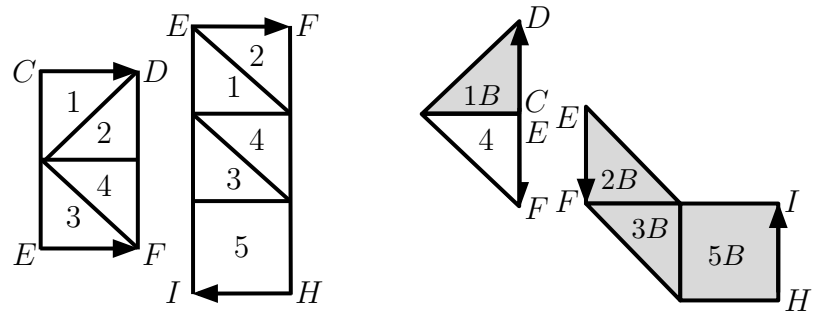
\end{center}

To construct a 5 half-twist paper  M\"obius band with aspect ratio 5, simply glue a 2-half-twist paper annulus with aspect ratio 2 to a flat 3-half twist paper M\"obius band with aspect ratio 3.

\begin{const}[5 half-twist paper M\"obius band] \label{const:5-HTw-mb} On the left in Figure~\ref{fig:UnK-5HTw}, we find a 2-half-twist paper annulus from Construction~\ref{const:2-HTw-an} and a flat 3-half twist paper M\"obius band from Construction~\ref{const:3-HTw-mb}. These separate models are joined along edge $EF$ to create one model.  We add half-twists to each model following Constructions~\ref{const:2-HTw-an} and~\ref{const:3-HTw-mb}, the results are shown on the right in Figure~\ref{fig:UnK-5HTw}.  In the next step of the construction, we join the two edges $EF$ together. Looking at the half of the model from the 3 half-twist paper M\"obius band, we fold triangles 1 and 2 down over triangles 3 and 4. This has the effect of flipping over the 2 half-twist paper annulus joined along edge $EF$. The result can be seen on the left in Figure~\ref{fig:UnK-5HTw-2}.  We complete the construction as follows. From the left side of the model, fold triangles 1 and 2 up and under triangles 3 and 4 so edge $CD$ lies underneath the model. From the right side of the model, fold square 5 to the right under the center four triangles. At this point, edge $CD$ and be joined to edge $HI$. The completed 5 half-twist paper M\"obius band can be seen on the right  in Figure~\ref{fig:UnK-5HTw-2}. \qed
\end{const}

\begin{center}
    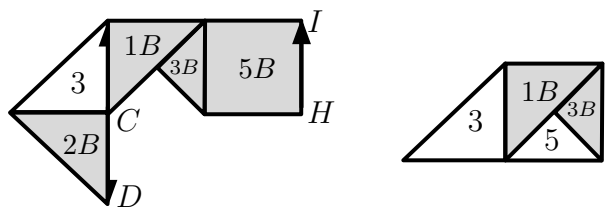
\begin{figure}[htpb]
    \begin{overpic}{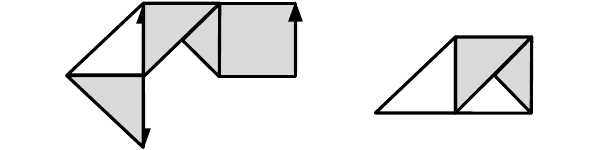}
    \put(19,15){$3$}   
    \put(26,20){\small{$1B$}}
    \put(18,7){\small{$2B$}}
     \put(32,17.5){\tiny{$3B$}}
  \put(41,17){\small{$5B$}}
    \put(25,10){$C$}
   \put(25,0){$D$}
      \put(50,11){$H$}
         \put(50,23){$I$}
    \put(78,14){\small{$1B$}}
     \put(84,12){\tiny{$3B$}}
  \put(81,7.5){$5$}  
   \put(71,10){$3$}      
    \end{overpic}
    \caption{A 5 half-twist paper M\"obius band with aspect ratio 5 is constructed once edge $CD$ is identified to edge $HI$.}
    \label{fig:UnK-5HTw-2}
    \end{figure}
\end{center}

Construction~\ref{const:5-HTw-mb} shows how to build a folded ribbon unknot (which is a topological M\"obius band) that has ribbon linking number~5. This can then be approximated as closely as we like, by a 5 half-twist paper M\"obius band. We have thus proved Proposition~\ref{prop:5tw-mb}: the infimal aspect ratio of a 5 half-twist paper M\"obius band is $\leq 5$.

Finally, we end this section by constructing a 6 half-twist paper annulus with aspect ratio 4. Here, we join together two flat 3 half-twist paper M\"obius bands together (after first removing the square pieces).

\begin{center}
    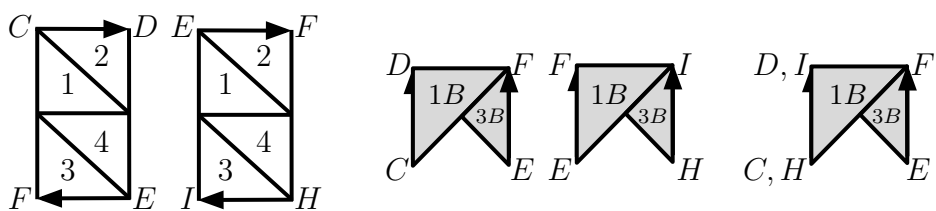
\begin{figure}[htpb]
    \begin{overpic}{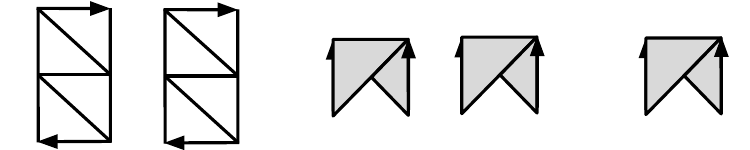}
   \put(2,18){$C$}
  \put(15,18){$D$}
   \put(7.5,12){$1$}
   \put(11,15){$2$}
    \put(7.5,3){$3$}
  \put(11,6){$4$}
    \put(15,0){$E$}
    \put(2,0){$F$}
    \put(19,18){$E$}
  \put(32,18){$F$}
   \put(24,12){$1$}
   \put(28,15){$2$}
    \put(24,3){$3$}
  \put(28,6){$4$}
    \put(32,0){$H$}
    \put(20,0){$I$}
    \put(46,11){\small{$1B$}}
     \put(51,9){\tiny{$3B$}}
    \put(41.5,3){$C$}
   \put(41.5,14){$D$}
      \put(54.55,3){$E$}
         \put(54.55,14){$F$}
    \put(63,11){\small{$1B$}}
     \put(68,9){\tiny{$3B$}}
    \put(58.5,3){$E$}
   \put(58.5,14){$F$}
      \put(72,3){$H$}
         \put(72,14){$I$}
    \put(88,11){\small{$1B$}}
     \put(92.5,9){\tiny{$3B$}}  
    \put(96,3){$E$}
   \put(96.5,14){$F$}
    \put(79,3){$C,H$}
     \put(80,14){$D,I$}
    \end{overpic}
    \caption{A 6 half-twist paper annulus with aspect ratio 4 is constructed from two 3 half-twist paper M\"obius bands.}
    \label{fig:UnK-6HTw}
    \end{figure}
\end{center}

\begin{const}[6 half-twist paper annulus]\label{const:6-HTw-an}
We start by taking two flat 3 half-twist paper M\"obius bands from Construction~\ref{const:3-HTw-mb}. We have removed the square labeled by 5 from each, thus the two rectangles each have aspect ratio 2 and are shown on the left in Figure~\ref{fig:UnK-6HTw}.  We add three half-twists to each model following  Construction~\ref{const:3-HTw-mb}. Thus triangle 2 is folded over triangle 1 to the left, and triangle 3 is folded over triangle 4 to the right. Then triangles 1 and 2 are folded down over triangles 3 and 4. This is shown in the center of Figure~\ref{fig:UnK-6HTw}.  These separate models are joined along edge $EF$ to create one model. In the final step, we fold along edge $EF$, then join edges $CD$ and $HI$ together. The completed 6 half-twist paper annulus can be seen on the right  in Figure~\ref{fig:UnK-6HTw}.  \qed
\end{const}

Constructions~\ref{const:4-HTw-an} and~\ref{const:6-HTw-an} show how to build folded ribbon unknots (which are topological annuli) that have ribbon linking numbers 2 and 3 respectively. Each of these can then be approximated as closely as we like by a 4 or 6 half-twist paper annulus. We have thus proved Proposition~\ref{prop:4-6-an}: the infimal aspect ratio of both a 4 and 6 half-twist paper annulus is $\leq 4$.

\section{Even twist knots}
\label{sect:even-twist}

\begin{const}[Folded ribbon $2n$-twist knots] \label{const:twist-even}
We start by using the same construction of half-twists as found in Construction~\ref{const:torus}. As shown in Figure~\ref{fig:accord-layer}, we start with two escape accordions with fold angle $\pi/3$. The first is labeled $CD$ and is made longer by adding additional $V$-units. The second is labeled $AB$ and is placed on top of $CD$ as in Figure~\ref{fig:accord-wrap}. We then add $2n$ half-wraps of $AB$ around $CD$ as shown on the left in Figure~\ref{fig:twist-start-even}. This creates $2n$ half-twists in the corresponding knot diagram.  In this figure, the short horizontal line on end $B$  denoted by $w_{2n+1}$ marks the end of the $2n$-th half-wrap. The the short horizontal line on end $D$ denoted by $Z$ marks the end of the $V$-units used in the second escape accordion.  We now create a clasp in five  steps.

\begin{center}
    \begin{figure}[htpb]
    \begin{overpic}{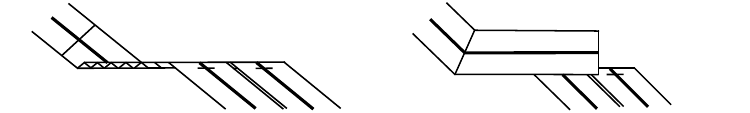}
    \put(13,12){$A$}
    \put(34,-1){$B$}
    \put(5,13){$C$}
    \put(41.5,-1){$D$}
    \put(24,7.5){\tiny{$w_{2n+1}$}}
    \put(33,7.25){\small{$Z$}}
    \put(80,8){$A$}
    \put(79.5,-1){$B$}
    \put(55,12.5){$C$}
    \put(86,-1){$D$}
    \end{overpic}
    \caption{On the left,  the end $B$ is wrapped around the thin accordion fold of $CD$ creating the $2n$ half-twists. On the right, end $A$ is folded to the right over end $B$.}
    \label{fig:twist-start-even}
    \end{figure}
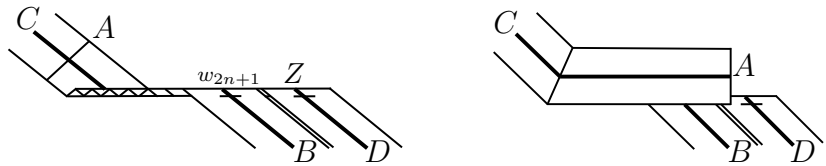
\end{center}

{\bf Step 1:} Fold end $A$ with fold angle $\pi/3$ over to the right, so that end $A$ lies over end $B$.  This is shown on the right in Figure~\ref{fig:twist-start-even}. The fold angle means that the lower edge of the ribbon near end $A$ is parallel to the lower edge of the accordion folds.

{\bf Step 2:} Fold end $B$ upward over end $A$ as shown on the left in Figure~\ref{fig:twist-clasp-even}. We arrange this fold so that the new fold line ends at the bottom edge of the ribbon near end $A$. The ribbon near end $B$ is perpendicular to the ribbon near end $A$, and thus the local geometry means the fold angle is $\pi/6$. 

{\bf Step 3:} Fold end $C$ with fold angle $\pi/3$ over to the right so that end $C$ lies over end $B$, as shown on the right in Figure~\ref{fig:twist-clasp-even}. Join the ends of $C$ and $A$ with a new fold line so this fold line lies next the side-edge of end $B$. 

\begin{center}
    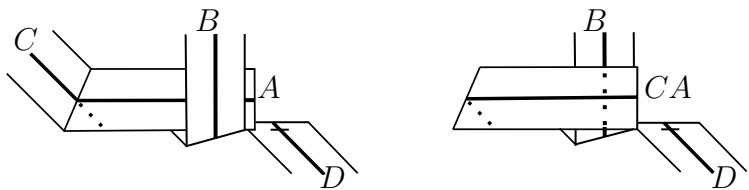
\begin{figure}[htpb]
  \begin{overpic}{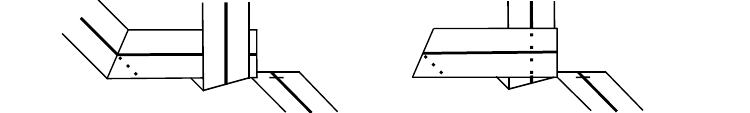}
    \put(34.5,8){$A$}
    \put(28,15){$B$}
    \put(9,13){$C$}
    \put(41,-1.5){$D$}
    \put(68.5,15){$B$}
    \put(75,8){$CA$}
    \put(82,-1.5){$D$}
    \end{overpic}
   \caption{On the left, end $B$ is folded up at angle $\pi/6$ perpendicular to end $A$. On the right, end $C$ is folded to the right to cover end $B$ and join with end $A$.}
    \label{fig:twist-clasp-even}
   \end{figure}
\end{center}

{\bf Step 4:} Fold end $D$ upwards with fold angle $\pi/6$ so that it lies parallel to end $B$ as shown on the left in Figure~\ref{fig:twist-join-even}. We arrange this fold so that the fold line replicates the geometry of the fold line in Step 2. 

{\bf Step 5:} Fold end $D$ with fold angle $\pi/2$ to the left so that it lies over end $CA$. Next, fold end $D$ up with fold angle $\pi/2$ so that end $D$ lies directly over end $B$. This is shown on the right in Figure~\ref{fig:twist-join-even}. We then join ends $D$ and $B$ together with a new fold line so that this fold line lies next to the side-edge of end $CA$.   The end result can be seen in Figure~\ref{fig:final-even}.
\qed
\end{const}

\begin{center}
    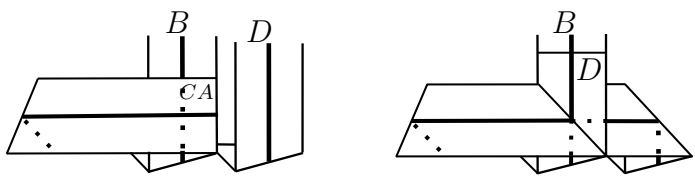
\begin{figure}[htpb]
   \begin{overpic}{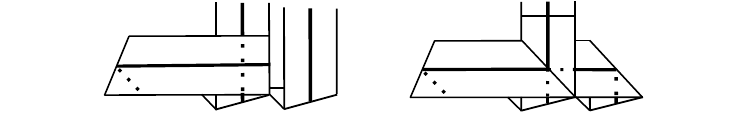}
    \put(30.5,15){$B$}
    \put(32,8){\tiny{$CA$}}
    \put(39,14){$D$}
    \put(71,15){$B$}
    \put(73.5,10){$D$}
  \end{overpic}
    \caption{On the left, end $D$ is folded up at angle $\pi/6$ and is perpendicular to end $CA$. On the right, end $D$ is folded $\pi/2$ to the left, then $\pi/2$ upwards, then joined to end $B$.}
   \label{fig:twist-join-even}
   \end{figure}
\end{center}

We are now ready to prove Theorem~\ref{thm:twist-even}: Any folded ribbon $2n$-twist knot type $K$ contains a folded ribbon knot $K_w$ such that its folded ribbonlength is $\Rib(K_w)=8\sqrt{3} + 2 +\epsilon$ for any $\epsilon>0$.

\begin{proof}[Proof of Theorem~\ref{thm:twist-even}]  We compute the folded ribbonlength of the $2n$-twist knot from Construction~\ref{const:twist-even}. Since we assumed the width of the folded ribbon $w=1$,  we just need to compute the length of the knot diagram.  We will refer to the distances in Figure~\ref{fig:final-even} throughout this proof. Note that the majority of these distances were previously computed in Theorem~\ref{thm:torus-bound} and Lemma~\ref{lem:tw-details}.

We compute the ribbonlength of the two pieces of ribbon labeled $AB$ and $CD$ used in Construction~\ref{const:twist-even}.  We measure each piece of ribbon starting vertex $E$ in Figure~\ref{fig:final-even}. 
We start with the length of the piece of ribbon labeled $AB$.  End $B$ follows the knot diagram from vertex $E$ to vertex $v_S$, then travels in a zig-zag fashion through the accordion folds and wraps to vertex $w_{2n+1}$. This distance traveled is 
$d(E,v_S) + d_K(v_S,v_E) + d_K(w_1,w_{2n+1})$.  Following Step 2 in Construction~\ref{const:twist-even}, end $B$ follows the knot diagram from vertex $w_{2n+1}$ to $Y$ to $X$ and ends at $R$. The geometry of the fold at $Y$ is the same as the geometry of the fold at vertex $N$ in Figures~\ref{fig:final} and \ref{fig:twist-compute}. Thus the distance along the knot from $w_{2n+1}$ to $X$ is $d_K(w_{2n+1},X) = d_K(M,P)=\frac{1}{2\sqrt{3}}$ from Lemma~\ref{lem:tw-details}.  We also see that $d(X,R)=1$ (as this is the width of the ribbon).  

\begin{center}
    \begin{figure}[htpb]
    \begin{overpic}{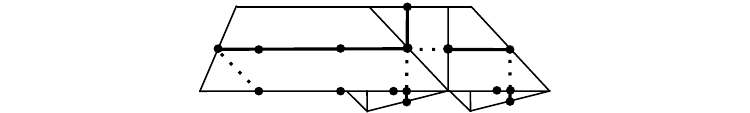}
    \put(26,8){$E$}
    \put(33,0.5){\small{$v_S$}}
    \put(33.5,9.5){$F$}
    \put(44,9.5){$G$}
    \put(37.5, 0.5){\footnotesize{$v_E=w_1$}}
    \put(47,4){\scriptsize{$w_{2n+1}$}}
    \put(64,3.5){\small{$Z$}}
    \put(53,15){$R$}
    \put(55, 9){$S$}
    \put(60.5,9){$T$}
    \put(69,8){$U$}
    \put(67, -1.5){\small{$W$}}
    \put(55,3.5){\scriptsize{$X$}}
    \put(68.5,3.5){\scriptsize{$V$}}
    \put(53.5, -1.5){\small{$Y$}}
    \end{overpic}
    \caption{Distances used in  the proof of Theorem~\ref{thm:twist-bound}. The figure is not to scale.}
    \label{fig:final-even}
    \end{figure}
\end{center}

Meanwhile, from Step 1 in Construction~\ref{const:twist-even}, we see end $A$ starts at vertex $E$ and ends at vertex $T$.  The distances are 
$$d(E,T) = d(E,F)+d(v_S,v_E)+ d(w_1,w_{2n+1})+d(w_{2n+1}, X) + d(S,T).$$
From Lemma~\ref{lem:tw-details} we know that $d(w_{2n+1}, X)= d(P,M)=\frac{1}{\sqrt{3}}-\frac{1}{2}$. Distance $d(S,T)=\frac{1}{2}$ as it is half the ribbon width.

Let us now find the length of the piece of ribbon labeled $CD$ in Construction~\ref{const:twist-even}. 
Just as with end $A$, end $C$ starts at vertex $E$ and ends at vertex $T$, which adds distance $d(E,T)$ again.
In a similar way to end $B$, end $D$ follows the knot diagram from vertex $E$ to vertex $v_S$, then travels in a zig-zag fashion through the accordion folds to vertex $w_{2n+1}$, then through a second escape accordion to vertex $Z$. The distance traveled is $d(E,v_S) + 2d_K(v_S,v_E) + d_K(w_1,w_{2n+1})$. 

Following Step 4 in Construction~\ref{const:twist-even}, end $D$ follows the knot diagram from vertex $Z$ to $W$ to $V$ along the knot. The geometry this fold is the same as the fold at vertex $Y$. Thus the distance traveled is the same $d_K(Z,V) = d_K(M,P)=\frac{1}{2\sqrt{3}}$. 
Finally, following Step 5, end $D$ follows the knot diagram from vertex $V$ to $U$ to $S$, and ends at $R$. The distance traveled is $d(V,U) + d(U,S) + d(S,R)$. The local geometry means that $d(U,S)=d(w_{2n+1}, Z)=d(v_S, v_E)$.   We also recall from Remark~\ref{rmk:length} that $d(V, U) = d(S,R)=\frac{1}{2}$.

When we put all of this together we find that the folded ribbonlength of our $2n$-twist knot is
\begin{align*}
\Rib(K) & =  2d(E,v_S) + 3d_K(v_S,v_E) + 2d_K(w_1,w_{2n+1}) + 2d(E,F) + 2d_K(M,P) 
\\ & \quad \quad  + 3d(v_S,v_E) + 2d(w_1,w_{2n+1})  + d(X,R) + 2d(P,M) + 4d(S,T)
\\ & = \frac{2}{\sqrt{3}} + \frac{12}{\sqrt{3}} + 4nd +\frac{1}{\sqrt{3}} +  \frac{1}{\sqrt{3}}+ \frac{6}{\sqrt{3}} + 2nd + 1 +  \left(\frac{2}{\sqrt{3}}-1\right) + 2
\\ & =  \frac{24}{\sqrt{3}} + 2 + 6nd= 8\sqrt{3} + 2 + 6nd \leq 15.86. 
\end{align*}
The last inequality and the theorem follow as distance $d$ can be made as small as we like.
\end{proof}

%

\bibliography{folded-ribbons}{}
\bibliographystyle{plain}


\end{document}